\documentclass[11pt]{article}
\usepackage{amsmath, amsthm}
\usepackage{amsfonts}
\usepackage{amssymb}
\usepackage[usenames,dvipsnames]{color}
\usepackage{graphicx}
\usepackage{hyperref}
\usepackage{bm}
\usepackage{listings}
\usepackage{geometry}
\usepackage{lipsum} 

\usepackage{color,calc}

\makeatletter
\definecolor{shade}{gray}{0.8}
        {
          \raggedright
        \setlength{\rightmargin}{\leftmargin}
        \setlength{\itemsep}{-12pt}
        \setlength{\parsep}{20pt}
        \begin{lrbox}{\@tempboxa}%
        \begin{minipage}{\linewidth-2\fboxsep}
        }%
        {
        \end{minipage}%
        \end{lrbox}%
        \fcolorbox{black}{shade}{\usebox{\@tempboxa}}\newline
        }%
\makeatother

\newtheorem{theorem}{Theorem}
\newtheorem{lemma}{Lemma}
\newtheorem{cor}{Corollary}

\newcommand{\iu}{\mathrm{i}} 

\renewcommand{\eqref}[1]{\hyperref[#1]{(\ref*{#1})}}

\newcommand{\dd}{\mathrm{d}}

\newcommand*{\pref}[1]{\hyperref[#1]{(\ref*{#1})}}
\newcommand*{\refpref}[2]{\hyperref[#2]{\ref*{#1}(\ref*{#2})}}

\newcommand{\R}{\mathbb{R}}

\title{Oscillatory attraction and repulsion from   a subset of the unit sphere or  hyperplane for isotropic stable L\'evy processes}

\author{ Mateusz Kwa\'{s}niki\thanks{Department of Pure Mathematics
Wrocław University of Science and Technology
ul. Wybrzeże Wyspiańskiego 27
50-370 Wrocław, Poland. Emaik: \texttt{Mateusz.Kwasnicki@pwr.edu.pl}},
\ Andreas E. Kyprianou\thanks{Department of Mathematical Sciences, University of Bath, Claverton Down, Bath, BA2 7AY, UK. Email: \texttt{a.kyprianou@bath.ac.uk} and \texttt{t.saizmaa@bath.ac.uk}}, 
\\ Sandra Palau\thanks{Department of Statistics and Probability, Instituto de Investigaciones en Matem\'aticas Aplicadas y en Sistemas, Universidad Nacional Aut\'onoma de M\'exico, M\'exico.
E-mail: \texttt{sandra@sigma.iimas.unam.mx}}
\
and 
Tsogzolmaa Saizmaa$^{\dagger,}$\thanks{ National University of Mongolia, Baga-toiruu, Sukhbaatar district, Ulaanbaatar, Mongolia. Email: \texttt{tsogzolmaa@num.edu.mn}}
\\
\\
{\it On the occasion of  Ron Doney's 80th birthday}
}
\begin{document}

\maketitle

\begin{abstract}
\noindent  Suppose that $\mathsf{S}$ is a closed set of the unit sphere $\mathbb{S}^{d-1} = \{x\in \mathbb{R}^d: |x| =1\}$ in dimension $d\geq2$, which has positive surface measure.   We construct the law of absorption of an isotropic stable L\'evy process in dimension $d\geq2$ conditioned to approach $\mathsf{S}$ continuously, allowing for the interior and exterior of $\mathbb{S}^{d-1}$ to be visited infinitely often. Additionally, we show that this 
 process is in duality with the underlying stable L\'evy process. 
 We can replicate the aforementioned results by similar ones in the setting that $\mathsf{S}$ is replaced by $\mathsf{D}$,  a closed bounded subset of the hyperplane $\{x\in\mathbb{R}^d : (x, v) = 0\}$ with positive surface measure, where $v$ is the unit orthogonal vector and where $(\cdot,\cdot )$ is the usual Euclidean inner product.
 Our results complement similar results of the authors \cite{KPS} in which the stable process was further constrained to attract to and repel from $\mathsf{S}$ from either the exterior or the interior of the unit sphere.

\medskip

\medskip

\end{abstract}

\section{Introduction}

Let $X=(X_t, t \geq 0)$ be a $d$-dimensional  stable  L\'evy process $(d\geq 2)$ with probabilities $(\mathbb{P}_x, x \in \mathbb{R}^d)$. This means that $X$ has c\`adl\`ag paths with stationary and independent increments as well as there existing an  $\alpha>0$ such that, for $c>0,$ and $x \in \mathbb{R}^d,$
under $\mathbb{P}_x$,
\[
\text{ the law of } (cX_{c^{-\alpha}t}, t \geq 0) \text{ is equal to  }\mathbb{P}_{cx}. 
 \]
 The latter is the property of so-called self-similarity.
It turns out that stable L\'evy processes necessarily have  $\alpha\in (0,2]$. The case $\alpha=2$ is that of   standard $d$-dimensional Brownian motion, thus has a continuous path. All other $\alpha\in(0,2)$ have no Gaussian component and are  pure jump processes. In this article we are specifically interested in phenomena that can only occur when jumps are present. We thus restrict ourselves henceforth to the setting $\alpha\in(0,2)$. 

\smallskip

Although Brownian motion is isotropic, this need not be the case in the stable case when $\alpha\in(0,2)$. {\it Nonetheless, we will restrict  to the  isotropic setting.} To be more precise, this means, for all orthogonal transformations $U:\mathbb{R}^d \mapsto \mathbb{R}^d$ and $x \in \mathbb{R}^d,$ 
\[
 \quad \textit{the law of} \quad (UX_t, t \geq 0) \textit{ under} \quad \mathbb{P}_x  \textit{ is equal to }  (X_t,t\geq 0) \textit{ under } \mathbb{P}_{Ux}.
 \]
For convenience, we will henceforth refer to $X$  as a {\it stable process}. 

\smallskip

As a L\'evy process, our stable   process of index $(0,2)$ has a characteristic triplet $(0,0,\Pi)$, where the jump measure $\Pi$ satisfies 
\begin{equation}
\Pi(B) = \frac{2^{\alpha} \Gamma((d+\alpha)/2)}{\pi^{d/2} |\Gamma(-\alpha/2)|} \int_B \frac{1}{|y|^{\alpha+d}} \ell_d({\rm d} y), \quad B \subseteq \mathcal{B}(\mathbb{R}^d),
\label{bigPi}
\end{equation}
where $\ell_d$ is $d$-dimensional Lebesgue measure\footnote{We will distinguish integrals with respect to one-dimensional Lebesgue measure as taking the form $\int\cdot\, \dd x$, where as higher dimensional integrals will always indicate the dimension, for example $\int \cdot\, \ell_d(\dd x)$.}. 
This is equivalent to identifying its characteristic exponent as 
\[
\Psi(\theta)=-\frac{1}{t}\log \mathbb{E} ({\rm e}^{\iu\theta \cdot X_t}) =|\theta|^{\alpha}, \quad \theta \in \mathbb{R}^d,
\]
where we write $\mathbb{P}$ in preference to $\mathbb{P}_0$.

\smallskip

In this article, we characterise the law of  a stable process conditioned to continuously approach  a closed   subdomain of the surface of a unit sphere, say $\mathsf{S}\subseteq\mathbb{S}^{d-1}= \{x\in \mathbb{R}^d: |x| =1\}$, which has non-zero surface measure. Moreover, our conditioning will allow the stable process to approach $\mathsf{S}$ by visiting the exterior and interior of $\mathbb{S}^{d-1}$ infinitely often. 
We note that when $\alpha\in(1,2)$, stable processes will hit the unit sphere with probability 1 and otherwise, when $\alpha\in(0,1]$ it hits the unit sphere with probability zero; see e.g. \cite{Port} or \cite{KALEA}. The aforesaid conditioning is thus only of interest when $\alpha\in(0,1]$. 

\smallskip

In addition to constructing the conditioned process, we develop an expression for the limiting point of contact on $\mathsf{S}$. Moreover, we show that, when time reversed from the strike point on $\mathsf{S}$, the resulting process can be described as nothing more than the  stable process itself.  The extreme cases that $\mathsf{S} = \mathbb{S}^{d-1}$ (the whole unit sphere) and $\mathsf{S} = \{\vartheta\}\in \mathbb{S}^{d-1}$ (a single point on the unit sphere) are included in our analysis, however,  we will otherwise insist that the Lebesgue surface measure of $\mathsf{S}$ is strictly positive.  In order to make our results pertinent, we restrict ourselves to the case that $d \geq 2$.


\smallskip

 It turns out that the methodology we use here is robust enough to cover a similar suite of results for the case of an isotropic stable process conditioned to a closed subdomain of an arbitrary $(d-1)$-dimensional hyperplane  in $\mathbb{R}^d$ that is orthogonal to an arbitrary unit-length vector $v\in\mathbb{R}^d$. 

\smallskip

Our results naturally complement those of  the recent paper \cite{KPS}, which considers a similar type of conditioning, albeit  requiring the stable process  to additionally remain  either inside or outside of the unit ball. 
Other related works include \cite{Phil}, who considered a real valued L\'evy process conditioned to continuously approach the boundary of the interval $[-1,1]$ from the outside. 


\section{Oscillatory attraction towards $\mathsf{S}$}\label{attraction}

Let $\mathbb{D}(\mathbb{R}^d)$ denote the space of c\'adl\'ag paths $\omega:[0,\infty)\to\mathbb{R}^d\cup\partial$ with lifetime $\zeta(\omega)=\inf\{s>0 :\omega(s)=\partial\}$, where $\partial$ is a cemetery point. The space $\mathbb{D}(\mathbb{R}^d)$ will be equipped with the Skorokhod topology, with its closed $\sigma$-algebra $\mathcal{F}$ and natural filtration $(\mathcal{F}_t, t\geq 0)$. 
The reader will note that we will also use 
a similar notion for $\mathbb{D}(E)$ later on in this text in the obvious way for an $E$-valued Markov process.
We will always work with $X = (X_t,t\geq0)$ to mean  the coordinate process defined on the space $\mathbb{D}(\mathbb{R}^d)$. Hence, the notation of the introduction indicates that $\mathbb{P}= (\mathbb{P}_x, x\in\mathbb{R}^d)$ is such that $(X,\mathbb{P})$ is our stable process.

\smallskip

We want to construct the law of the stable process conditioned to continuously limit to $\mathsf{S}\in\mathbb{S}^{d-1}$  
whilst visiting both $\mathbb{B}_d: = \{x\in\mathbb{R}^d: |x| < 1\}$ and  $ \bar{\mathbb{B}}_d^c : =\mathbb{R}^{d}\setminus \bar{\mathbb{B}}_d$ infinitely often at arbitrarily small times prior to 
striking $\mathsf{S}$.
We shall denote the associated probabilities by $\mathbb{P}^\mathsf{S}= (\mathbb{P}^\mathsf{S}_x, x\in\mathbb{R}^d)$. For a more precise definition of what is meant by this form of conditioning, 
let us introduce the stopping times,
\begin{equation}
\tau_\beta = \inf\{t>0: \beta^{-1}<|X_t|<\beta\},\qquad  \text{ for } \beta >1.
\label{taubeta}
\end{equation}
Whenever it is well defined, 
we will write, for $t\geq 0$, $\Lambda\in\mathcal{F}_t$ and $x\not\in\mathsf{S}$, 
\begin{equation}
\mathbb{P}^\mathsf{S}_x(\Lambda, \, t<\zeta) = \lim_{\beta\to1}
\lim_{\varepsilon \to0}\mathbb{P}_x\left(\Lambda, t<\tau_{\beta}
\big| 
\, \tau_{\mathsf{S}_\varepsilon}<\infty \right),
\label{conditionlimit}
\end{equation}
where 
\[
\tau_{\mathsf{S}_\varepsilon} = \inf\{t>0: X_t\in \mathsf{S}_\varepsilon\}\,\text{ and  }\,
\mathsf{S}_\varepsilon : = \{x\in  \mathbb{R}^d: 1-\varepsilon\leq |x|\leq  1+\varepsilon \text{ and }\arg(x)\in \mathsf{S}\}.
\]
Our first  main result clarifies that $(X, \mathbb{P}^\mathsf{S})$ is indeed well defined. In the theorem below, and thereafter, we will understand $\sigma_1$ to mean the Lebesgue surface measure on $\mathbb{S}^{d-1}$ normalised to have unit mass, i.e. $\sigma_1(\mathbb{S}^{d-1}) =1$.

\begin{theorem}\label{main}
Suppose that $\alpha\in(0,1]$ and the closed set $\mathsf{S}\subseteq\mathbb{S}^{d-1}$ is such that  $\sigma_1(\mathsf{S})>0$. 
For $\alpha\in(0,1]$, the process $(X, \mathbb{P}^\mathsf{S})$ is well defined such that 
\begin{equation}
\left.\frac{\dd \mathbb{P}^\mathsf{S}_x}{\dd \mathbb{P}_x}\right|_{\mathcal{F}_t} = \frac{H_\mathsf{S}(X_t)}{H_\mathsf{S} (x)}, \qquad t\geq 0, x\not\in \mathsf{S},
\label{doobH}
\end{equation}
where
\[
H_\mathsf{S} (x) =
\int_\mathsf{S} |x-\theta|^{\alpha - d} \sigma_1(\dd \theta), \qquad x\not\in\mathsf{S}.
\]
\end{theorem}

Although excluded from the conclusion of Theorem \ref{main}, it is worth dwelling for a moment on   the  extreme case  $\mathsf{S}=\{\theta\}$, for $\theta\in\mathbb{S}^{d-1}$. It has been shown in \cite{KRSg} that,   when $\alpha\in(0,1)$,  conditioning a stable process to continuously limit to a point (which, by stationary and independent increments, can always be arranged to be $\theta\in\mathbb{S}^{d-1}$) results in a family of probability measures $(\mathbb{P}^{\{\theta\}}_x, x\neq \theta) $ which can be identified via a Doob $h$-transform with $h_\theta(x) = |x- \theta|^{\alpha -d}$.  Although the sense in which the conditioning is performed cannot be contextualised via \eqref{conditionlimit}, we see that the resulting $h$-transformation is consistent with the use of the harmonic function  $H_\mathsf{S}$.

The way in which we will prove Theorem \ref{main} will be to prove the following subtle result which establishes the leading order behaviour of the probability of hitting the set 
$
\mathsf{S}_\varepsilon 
$.

\begin{theorem}\label{main2}
Let $\mathsf{S}\subseteq\mathbb{S}^{d-1}$ be a closed subset such that  $\sigma_1(\mathsf{S})>0$. 
\begin{itemize}
\item[(i)] Suppose $\alpha\in(0,1)$. For $x\not\in\mathsf{S}$,
\begin{equation}
\lim_{\varepsilon\to 0}\varepsilon^{\alpha-1} \mathbb{P}_x(\tau_{\mathsf{S}_\varepsilon}<\infty  )
= 
\textcolor{black}{2^{1-2\alpha}\frac{\Gamma((d+\alpha-2)/2)}{ \pi^{d/2}\Gamma(1-\alpha) }\frac{ \Gamma((2-\alpha)/2)}{ \Gamma(2-\alpha)}} H_\mathsf{S}(x).
\label{gammarecursion}
\end{equation}
\item[(ii)]When $\alpha = 1$, we have that, for $x\not\in\mathsf{S}$,
\begin{equation}
\lim_{\varepsilon\to 0}\ |\log \varepsilon| \ \mathbb{P}_x(\tau_{\mathsf{S}_\varepsilon}<\infty  )
= 
\frac{\Gamma((d-1)/2)}{ \pi^{(d-1)/2}} H_\mathsf{S}(x).
\label{gammarecursiona = 1}
\end{equation}
\end{itemize}
\end{theorem}
Theorem \ref{main2} also gives us the opportunity to understand the strike position of the the conditioned stable process.
Indeed, let $\mathsf{S}' $ be a closed subset of $\mathsf{S}$. Define $\mathsf{S}'_\varepsilon = \{r\theta \colon r \in (1-\varepsilon,1+\varepsilon)\text{ and }\theta\in\mathsf{S}'\}$ and $\tau_{\mathsf{S}'_\varepsilon} := \inf\{t>0 \colon X_t \in \mathsf{S}'_\varepsilon\}$. Then, $\{\tau_{\mathsf{S}'_\varepsilon} < \infty\} \subseteq \{\tau_{\mathsf{S}_\varepsilon} < \infty\}$ and thanks to Theorem \ref{main2}, when $\alpha\in(0,1)$, we have 
\begin{align*}
\lim_{\varepsilon \rightarrow 0} \mathbb{P}_x (\tau_{\mathsf{S}'_\varepsilon} < \infty  | \tau_{\mathsf{S}_\varepsilon} < \infty )&
= \lim_{\varepsilon \rightarrow 0}
 \frac{\varepsilon^{\alpha-1}\mathbb{P}_x(\tau_{\mathsf{S}'_\varepsilon} < \infty
 )}{\varepsilon^{\alpha-1}\mathbb{P}_x( \tau_{\mathsf{S}_\varepsilon} < \infty)}
 = \frac{H_{\mathsf{S}'}(x)}{H_{\mathsf{S}}(x)} \qquad x\not\in\mathsf{S}.
\end{align*}
A similar statement also holds when $\alpha = 1$ by changing the scaling in $\varepsilon$ to $|\log \varepsilon|$. This gives us the following result.

\begin{cor}\label{main3} For a closed $\mathsf{S}\subseteq\mathbb{S}^{d-1}$ such that  $\sigma_1(\mathsf{S})>0$ and $\alpha\in(0,1]$,  we have that for all closed $\mathsf{S}'\subseteq\mathsf{S}$,
\begin{equation}
\mathbb{P}^\mathsf{S}_x(X_{\zeta-} \in \mathsf{S}') = \frac{H_{\mathsf{S}'}(x)}{H_{\mathsf{S}}(x)} \qquad x\not\in\mathsf{S}.
\label{strikepoint}
\end{equation}
\end{cor}

In light of the above Corollary, it is worth remarking that we can also see  the probabilities $\mathbb{P}^\mathsf{S}$ as the result of first conditioning to continuously hit $\mathbb{S}^{d-1}$ and then conditioning the strike point to be in $\mathsf{S}$. Indeed, we note that, for $A\in \mathcal{F}_t$ and $t\geq 0 $, 
\begin{align*}
\mathbb{P}^{\mathbb{S}^{d-1}}_x(A| X_{\zeta-}\in \mathsf{S}) &= \mathbb{E}_x^{\mathbb{S}^{d-1}}\left[\mathbf{1}_A \frac{\mathbb{P}^{\mathbb{S}^{d-1}}_{X_t} (X_{\zeta-} \in \mathsf{S})}{\mathbb{P}^{\mathbb{S}^{d-1}}_x (X_{\zeta-} \in \mathsf{S})}\right]\\
&= \mathbb{E}_x\left[\mathbf{1}_A \frac{H_{\mathbb{S}^{d-1}} (X_t)}{H_{\mathbb{S}^{d-1}} (x)}
\frac{H_\mathsf{S}(X_t)}{H_{\mathbb{S}^{d-1}} (X_t)}\frac{H_{\mathbb{S}^{d-1}} (x)}{H_\mathsf{S}(x)}\right]\\
& =  \mathbb{E}_x\left[\mathbf{1}_A \frac{H_\mathsf{S}(X_t)}{H_\mathsf{S}(x)}\right]\\
&=\mathbb{P}^\mathsf{S}_x(A).
\end{align*} 
Moreover, by shrinking $\mathsf{S}'\subseteq\mathsf{S}\subseteq\mathbb{S}^{d-1}$ to a singleton $\theta\in\mathbb{S}^{d-1}$, 
one can similarly show that  
\[
\mathbb{P}^{\mathsf{S}}_x(A| X_{\zeta-} = \theta) =  \mathbb{P}^{\{\theta\}}_x(A ).
\]
This has the flavour of a Williams' type decomposition that was shown for general L\'evy processes conditioned to stay positive and subordinators conditioned to remain in an interval; see e.g \cite{C96} and \cite{KRS}.

%
%

\section{Oscillatory repulsion from $\mathsf{S}$ and duality}\label{repulsion}

Roughly speaking, we want to describe what we see when we time reverse the process $(X, \mathbb{P}^\mathsf{S})$  from its strike point on $\mathsf{S}$, i.e. its so-called dual process. Such a process will necessarily avoid visiting $\mathsf{S}$. Recalling that, for $\alpha\in(0,1]$, the stable process hits spherical surfaces with probability zero (cf. \cite{KALEA, Port}), a heuristic guess for the aforesaid dual process is the stable process itself (see Figure \ref{fig}). This turns out to be precisely the case. In order to make this rigorous, we will use the language of Hunt-Nagasawa duality for Markov processes.

\begin{figure}[h!]
\begin{center}
\includegraphics[width= 0.5\textwidth]{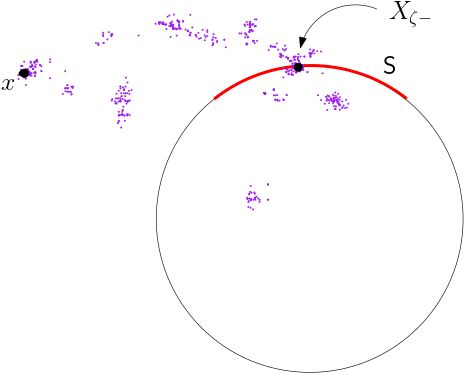}
\caption{\rm The process $(X,\mathbb{P}^\mathsf{S})$ when time reversed is stochastically equal in law to $(X,\mathbb{P})$.}
\label{fig}
\end{center}

\end{figure}

 Suppose that  $Y = (Y_t, t\leq \zeta)$ with probabilities ${\rm\texttt{P}}_x$, $x\in E$, is a regular Markov process on an open domain $E \subseteq \mathbb{R}^d$ (or more generally, a locally compact Hausdorff space with countable base), with cemetery state $\Delta$ and killing time $\zeta=\inf\{t>0: Y_t = \Delta\}$. Let us additionally write   ${\texttt P}_\nu = \int_{E}\nu(\dd a){\texttt P}_a$, for any probability measure $\nu$ on the state space of $Y$.
	
	\smallskip
	
	Suppose that $\mathcal{G}$ is the $\sigma$-algebra generated by $Y$ and  write $\mathcal{G}({\texttt P}_\nu)$ for its completion by the null sets of ${\texttt P}_\nu$. Moreover, write $\overline{\mathcal G} =\bigcap_{\nu} \mathcal{G}({\texttt P}_\nu)$, where the intersection is taken over all probability measures on the state space of $Y$, excluding the cemetery state.
		A finite  random time $\texttt{k}$ is called an $L$-time (generalized last exit time) if, given a coordinate process $\omega = (\omega_t, t\geq0)$ on $\mathbb{D}(E)$, 
\begin{itemize}
	\item[(i)] $\texttt{k}$ is measurable in $\overline{\mathcal G}$, and  $\texttt{k}\leq \zeta$  almost surely with respect to ${\texttt P}_\nu$, for all $\nu$,
	\item[(ii)] $\{s<\texttt{k}(\omega)-t\}=\{s<\theta_t\circ\texttt{k}\}$ for all $t,s\geq 0$,
\end{itemize}	
where $\theta_t$ is the Markov shift of $\omega$ to time $t$.
The most important examples of $L$-times are killing times and last  exit times from closed sets.

\begin{theorem}\label{Naga} Suppose that $\alpha\in(0,1]$. For a given closed set $\mathsf{S}\subset\mathbb{S}^{d-1}$ with $\sigma_1(\mathsf{S})>0$, write
\begin{equation}
\nu(\dd a) : = \frac{\sigma_1(\dd a)}{\sigma_1(\mathsf{S})}, \qquad a \in \mathsf{S}.
\label{closed}
\end{equation}
For every $L$-time  $\emph{\texttt{k}}$ of $(X, \mathbb{P})$, the process
$(
X_{(\emph{\texttt{k}} - t)-}, t< \emph{\texttt{k}}
)$
under $\mathbb{P}_\nu$ is a time-homogeneous Markov process whose transition probabilities agree with those of $(X,\mathbb{P}^\mathsf{S})$.
\end{theorem}

{
\section{The setting of a subset in an $\mathbb{R}^{d-1}$ hyperplane}
As alluded to in the introduction, the methods use in Section \ref{attraction} and \ref{repulsion} are robust  enough  to deal with the setting of an arbitrary $(d-1)$-dimensional hyperplane in $\mathbb{R}^d$. Without loss of generality, we can 
describe such a hyperplane with unit orthogonal vector $v\in \mathbb{S}^{d-1}$ via 
\[
\mathbb{H}^{d-1}  =\{x\in\mathbb{R}^{d}: (x, v ) = 0\}, 
\]
where $(\cdot,\cdot )$ is the usual Euclidean inner product. Henceforth, we will assume that $v\in\mathbb{S}^{d-1}$ is given, as it otherwise plays no role in the forthcoming. We are interested in defining the law of the stable process conditioned to hit $\mathsf{D}\subseteq\mathbb{H}^{d-1}$ in a similar spirit to the discussion in Section \ref{attraction}.
\smallskip

To this end, let us  define 
\[
\kappa_\beta = \inf\{t>0: -\beta<(v, X_t)<\beta\},\qquad \text{ for } \beta >0.
\]
Whenever it is well defined, 
we will write, for $t\geq 0$, $\Lambda\in\mathcal{F}_t$ and $x\not\in\mathsf{D}$, 
\begin{equation}
\mathbb{P}^\mathsf{D}_x(\Lambda, \, t<\zeta) = \lim_{\beta\to0}
\lim_{\varepsilon \to0}\mathbb{P}_x\left(\Lambda, t<\kappa_{\beta}
\big| 
\, \tau_{\mathsf{D}_\varepsilon}<\infty \right),
\label{conditionlimithyper}
\end{equation}
where 
\[
\tau_{\mathsf{D}_\varepsilon} = \inf\{t>0: X_t\in \mathsf{D}_\varepsilon\}\qquad\text{ and  }\qquad
\mathsf{D}_\varepsilon : = \{x\in  \mathbb{R}^d: -\varepsilon\leq (v, x)\leq  \varepsilon \text{ and }\hat{x}\in \mathsf{D}\}.
\]
Here $\hat x$ denotes the orthogonal projection of $x$ onto $\mathbb{H}^{d-1}$; in other words. $\hat{x} = x - v(v,x).$
We can gather the analogous conclusions of Theorems \ref{main}, \ref{main2}, \ref{Naga} and Corollary \ref{main3} into one theorem.

\begin{theorem}\label{main5}
Suppose that $\alpha\in(0,1]$ and the closed and bounded set $\mathsf{D}\subseteq\mathbb{H}^{d-1}$ is such that  $0<\ell_{d-1}(\mathsf{D})<\infty$, where we recall that  $\ell_{d-1}$ is $(d-1)$-dimensional Lebesgue measure.
\begin{itemize}
\item[(i)] Suppose $\alpha\in(0,1)$. For $x\not\in\mathsf{D}$,
\begin{equation}
\lim_{\varepsilon\to 0}\varepsilon^{\alpha-1} \mathbb{P}_x(\tau_{\mathsf{D}_\varepsilon}<\infty  )
= 
2^{1-\alpha}\pi^{-(d-2)/2}\frac{\Gamma(\frac{d-2}{2})\Gamma(\frac{d-\alpha}{2}) \Gamma(\frac{2-\alpha}{2})^2}{\Gamma(\frac{1-\alpha}{2})\Gamma(\frac{d-1}{2})\Gamma(2-\alpha)}M_\mathsf{D}(x),
\label{gammarecursion2}
\end{equation}
where
\[
M_\mathsf{D} (x) =
\int_\mathsf{D} |x-y|^{\alpha - d} \ell_{d-1}(\dd y), \qquad x\not\in\mathsf{D}.
\]

\item[(ii)] Suppose $\alpha = 1$. For $x\not\in\mathsf{D}$,
\begin{equation}
\lim_{\varepsilon\to 0} \ |\log \varepsilon|\ \mathbb{P}_x(\tau_{\mathsf{D}_\varepsilon}<\infty  )
= 
\frac{\Gamma(\frac{d-2}{2})}{\pi^{(d-2)/2}}M_\mathsf{D}(x),
\label{gammarecursion2a=1}
\end{equation}

\item[(iii)]
 The process $(X, \mathbb{P}^\mathsf{D})$ is well defined such that 
\begin{equation}
\left.\frac{\dd \mathbb{P}^\mathsf{D}_x}{\dd \mathbb{P}_x}\right|_{\mathcal{F}_t} = \frac{M_\mathsf{D}(X_t)}{M_\mathsf{D} (x)}, \qquad t\geq 0, x\not\in \mathsf{D}.
\label{doobM}
\end{equation}
\item[(iv)]We have for all closed $\mathsf{D}'\subseteq\mathsf{D}$,
\begin{equation}
\mathbb{P}^\mathsf{D}_x(X_{\zeta-} \in \mathsf{D}') = \frac{M_{\mathsf{D}'}(x)}{M_{\mathsf{D}}(x)} \qquad x\not\in\mathsf{D}.
\label{strikepoint2}
\end{equation}
\item[(v)] Write
$
\nu(\dd a) : = {\ell_{d-1}(\dd a)}/{\ell_{d-1}(\mathsf{D})}$, $ a \in \mathsf{D}.
$
For every $L$-time  $\emph{\texttt{k}}$ of $(X, \mathbb{P})$, the process
$(
X_{(\emph{\texttt{k}} - t)-}, t< \emph{\texttt{k}}
)$
under $\mathbb{P}_\nu$ is a time-homogeneous Markov process whose transition probabilities  agree with those of $(X,\mathbb{P}^\mathsf{D})$.
\end{itemize}
\end{theorem}
%
Roughly speaking, Theorem \ref{main5}  are to be expected as, following the ideas of \cite{Luks} one may map $\mathbb{S}^{d-1}$ onto $\mathbb{H}^{d-1}$ via a standard sphere inversion  transformation, which, thanks to the Riesz--Bogdan--\.Zak transform, also transforms  the paths of the stable processes into that of a $h$-transformed stable processes; see \cite{BZ}. The proofs we have given below, however, are direct nonetheless, following similar steps to those of Theorems \ref{main}, \ref{main2} and \ref{Naga}, as well as Corollary \ref{main3}.

\section{Heuristic for the proof of Theorem \ref{main2}}
Let us begin with a sketch of the proof of Theorem \ref{main2}. We start by recalling an identity that is known in quite a general setting from the potential analysis literature; see for example Section 13.11 of \cite{CW} and Section VI.2 of \cite{BGbook}.  Suppose that $A$ is a bounded closed set and let 
 $\tau_A = \inf\{t>0: X_t \in A\}$. Let $\mu_A$ be a finite measure supported on $A$, which is absolutely continuous with respect to Lebesgue measure and define its potential by
\[
U\mu_A(x) : = \int_{A} |x-y|^{\alpha - d}\mu_A(\dd y) \qquad x\in \mathbb{R}^d, 
\]
On account of the fact that $\mu_A$ is absolutely continuous, recalling that $|x|^{\alpha-d}$ is the potential of the stable process issued from the origin, stationary and independent increments allows us to identify we can identity 
\[
U\mu_A(x) = \int_{A} |x-y|^{\alpha - d}m_A( y)\ell_d(\dd y)=\mathbb{E}_x\left[\int_0^\infty m_A(X_t)\dd t\right], \qquad x\notin A,
\]
where $m_A$ is the density of $\mu_A$ with respect to Lebesgue measure, $\ell_d$. 
As the support of $\mu_A$ is precisely $A$, we must have $m_A(y) = 0$ for all $y\notin A$. As such, the Strong Markov Property tells us that
\begin{equation}
  U\mu_A(x) =\mathbb{E}_x\left[\mathbf{1}_{\{\tau_{A} < \infty\}}\int_{\tau_A}^\infty m_A(X_t)\dd t \right]
  =\mathbb{E}_x \left[U{\mu}_A (X_{\tau_A}) \mathbf{1}_{\{\tau_{A} < \infty\}}\right], \qquad x\notin A.
  \label{preuse}
\end{equation}
Note, the above equality is also true when $x\in A$ as, in that case,  $\tau_A = 0$.
\smallskip

Replacing $\tau_A$ by a general stopping time $\tau$ in the above calculation changes the first equality in \eqref{preuse} to an inequality, thus giving the excessive property 
\begin{equation}
\label{excessive}
 U\mu_A(x) \geq \mathbb{E}_x \left[U{\mu}_A (X_{\tau}) \mathbf{1}_{\{\tau < \infty\}}\right], \qquad x\in \mathbb{R}^d.
 \end{equation} This family of inequalities together with the Strong Markov Property easily gives us the classical result that $(U\mu_A(X_t), t\geq 0)$ is a supermartingale.  

\smallskip

Let us now suppose that $\mu$ can be constructed in such a away that it is supported on $A$ such that, for all $x\in A$,  $U\mu(x) = 1$. We then recover 
the corollary  to Theorem 1 in Chapter 5 of \cite{CW}, see also equation (21) in the same chapter, which states that 
\[
\mathbb{P}_x(\tau_A<\infty) = U\mu(x), \qquad x\not\in A.
\]

Returning to the problem at hand, we can use the principals above to develop a 
a `guess and verify' approach to the proof, in particular, since we are not chasing an exact formula for $\mathbb{P}_x(\tau_{\mathsf{S}_\varepsilon}<\infty) $, but rather the  asymptotic leading order behaviour. Indeed, suppose we can `guess' a measure, say $\mu_\varepsilon$, supported on $\mathsf{S}_\varepsilon$, such that
\begin{equation}
U\mu^{\mathsf{S}}_\varepsilon (x) = 1+o(1), \qquad x\in \mathsf{S}_\varepsilon \text{ as }\varepsilon\to 0,
\label{o(1)}
\end{equation}
so that 
\begin{equation}
(1+o(1))\mathbb{P}_x(\tau_{\mathsf{S}_\varepsilon} <\infty) = U\mu^{\mathsf{S}}_\varepsilon (x), \qquad x\not\in \mathsf{S}_\varepsilon,
\label{o(1)2}
\end{equation}
then this would be a good basis from which to draw out the the leading order decay in $\varepsilon$, especially if our guess of $\mu_\varepsilon$ is such that $U\mu_\varepsilon$ is tractable.

\smallskip

In one dimension,  we know from Lemma 1 of \cite{PS}, that for a one-dimensional symmetric stable process, 
\begin{equation}
\int_{-1}^1 |x-y|^{\alpha - 1} (1-y)^{-\alpha/2}(1+y)^{-\alpha/2}\dd y = 1, \qquad x\in[-1,1].
\label{Simon}
\end{equation}
We can use this to build a reasonable choice of $\mu^{\mathsf{S}}_\varepsilon$. Indeed, writing $X = |X|\arg(X)$, 
 when $X$ begins in the neighbourhood of  $\mathsf{S}$, then $|X|$ begins in the neighbourhood of 1 and $\arg(X)$, essentially, from within $\mathsf{S}$. On short-time scales and short-range, the  time change $|X|$ behaves similarly to a one-dimensional stable process. Moreover, $\arg(X)$ is an isotropic process.
A reasonable guess for    $\mu^{\mathsf{S}}_\varepsilon$ would be to base it on the measure  
\begin{equation}
\mu_\varepsilon (\dd y) = c_{\alpha, d} (|y|-(1-\varepsilon))^{-\alpha/2} (1+\varepsilon-|y|)^{-\alpha/2} \ell_d(\dd y),
\label{mueps}
 \end{equation}
restricted to $\mathsf{S}_\varepsilon$, where we recall $\ell_d$ is $d$-dimensional Lebesgue measure and  $c_{\alpha, d}$ is a constant to be determined so that \eqref{o(1)} holds. As we will shortly see, when $\alpha \in(0,1)$, the constant $c_{\alpha, d}$ does not depend on $\varepsilon$, however, when $\alpha =1$, in order to respect \eqref{o(1)} we need to make it depend on $\varepsilon$.

\section{Proof of Theorem \ref{main2} (i)}
As alluded to in the previous section, we will work with the guess $\mu^{\mathsf{S}}_\varepsilon $ given by \eqref{mueps}. In order to show \eqref{o(1)}, we will take advantage of some of the symmetric features of $\mu_\varepsilon$, when seen as a measure over $\mathbb{S}^{d-1}_\varepsilon = \{x\in\mathbb{R}^{d} : 1-\varepsilon\leq |x|\leq 1+\varepsilon\}$. In particular, writing $\mu^{(1)}_\varepsilon$ as $\mu_\varepsilon$ restricted to $\mathbb{S}^{d-1}_\varepsilon$ and $\mu_\varepsilon^{(2)}$ as $\mu_\varepsilon$ restricted to $\hat{\mathsf{S}}_\varepsilon := \mathbb{S}^{d-1}_\varepsilon\backslash \mathsf{S}_\varepsilon$, we have the obvious difference.
\begin{equation}
U\mu^{\mathsf{S}}_\varepsilon(x) = U\mu^{(1)}_\varepsilon (x) - U\mu^{(2)}_\varepsilon(x) \qquad x\in \mathsf{S}_\varepsilon.
\label{differencee}
\end{equation}
Moreover, we would like to introduce 
\[
\mu_{\varepsilon,\delta}^{(2)}: =  \mu_\varepsilon|_{\hat{\mathsf{S}}_{\varepsilon}^{\delta}}
\]
where $\hat{\mathsf{S}}_{\varepsilon}^{\delta} =\mathbb{S}^{d-1}_\varepsilon \backslash {\mathsf{S}}_{\varepsilon}^{\delta}$ and 
\[
{\mathsf{S}}_{\varepsilon}^{\delta}:= \{x \in \mathbb{R}^d \colon 1-\varepsilon<|x|<1+\varepsilon  \text{ and } \arg(x)\in\mathsf{S}^\delta\},\text{ where  } \mathsf{S}^\delta = \{x\in\mathbb{S}^{d-1}:  \text{dist}(\arg(x),\mathsf{S}) <\delta \},
\]
 for some small $\delta>0$, which, in due course, will depend on $\varepsilon$. Note that, since $\mathsf{S}$ is closed, $\mathsf{S}^\delta$ (resp. ${\mathsf{S}}_{\varepsilon}^{\delta}$) shrinks to $\mathsf{S}$ (resp. ${\mathsf{S}}_\varepsilon$) when $\delta \rightarrow 0$. Then, we also have that 
\begin{equation}
U\mu^{\mathsf{S}^\delta}_\varepsilon(x) = U\mu^{(1)}_\varepsilon (x)  - U\mu^{(2)}_{\varepsilon,\delta} (x) \qquad x\in \mathsf{S}_\varepsilon.
\label{differencee2}
\end{equation}
The estimate \eqref{differencee2} will be useful for a certain lower bound that will give us what we need to prove Theorem \ref{main2}. We need to prove two technical   lemmas first. 
The first one deals with the term $U\mu^{(1)}_\varepsilon$.

\begin{lemma}\label{unif}
Suppose that we choose 
\[
c_{\alpha, d} =\frac{\Gamma((d+\alpha-2)/2)}{2^\alpha \pi^{d/2}\Gamma(1-\alpha) \Gamma((2-\alpha)/2)}.
\] 
Then,
\[
\lim_{\varepsilon\to0}\sup_{x\in\mathbb{S}^{d-1}_\varepsilon}|U\mu^{(1)}_\varepsilon(x) -1|=0
\]
as $\varepsilon\to0$.
\end{lemma}
\begin{proof}
Appealing to \eqref{3.665}, we have, for $x\in\mathbb{S}^{d-1}_\varepsilon$,
\begin{align}
   &U\mu_\varepsilon^{(1)} (x) \notag\\
   &= c_{\alpha, d}\int_{\mathbb{S}_\varepsilon^{d-1}} |x-y|^{\alpha-d} (|y|-(1-\varepsilon))^{-\alpha/2} (1+\varepsilon-|y|)^{-\alpha/2} \ell_d(\dd y) \notag \\
    &=  \frac{2c_{\alpha, d}\pi^{(d-1)/2}}{\Gamma((d-1)/2)} \int_{1-\varepsilon}^{1+\varepsilon} \frac{r^{d-1} }{(r-(1-\varepsilon))^{\alpha/2} (1+\varepsilon-r)^{\alpha/2}} \dd r \int_0^\pi \frac{\sin^{d-2}\theta \dd\theta}{(|x|^2-2|x|r\cos\theta+r^2)^{(d-\alpha)/{2}}} \notag \\
     &= \frac{2c_{\alpha, d}\pi^{d/2}}{\Gamma(d/2)} |x|^{\alpha-d} \int_{1-\varepsilon}^{|x|} 
     \frac{ {_2}F_1 \Big(\frac{d-\alpha}{2}, 1-\frac{\alpha}{2}; \frac{d}{2}; ({r}/{|x|})^2\Big)r^{d-1} }{(r-(1-\varepsilon))^{\alpha/2} (1+\varepsilon-r)^{\alpha/2} }  \dd r \notag \\
    &\hspace{4cm}+  \frac{2c_{\alpha, d}\pi^{d/2}}{\Gamma(d/2)} \int_{|x|}^{1+\varepsilon} \frac{ {_2}F_1 \Big(\frac{d-\alpha}{2}, 1-\frac{\alpha}{2}; \frac{d}{2}; ({|x|}/{r})^2\Big) r^{\alpha-1} } {(r-(1-\varepsilon))^{\alpha/2} (1+\varepsilon-r)^{\alpha/2} }  \dd r. \label{firststep}
 \end{align}
 With a simple change of variables we can reduce this more simply to 
 \begin{align}
U\mu_\varepsilon^{(1)} (x) &=  \frac{2c_{\alpha, d}\pi^{d/2}}{\Gamma(d/2)} \int_{\frac{1-\varepsilon}{|x|}}^{1} 
     \frac{
     {_2}F_1 \Big(\frac{d-\alpha}{2}, 1-\frac{\alpha}{2}; \frac{d}{2}; r^2\Big)  r^{d-1} 
     }{
      \Big(r-\frac{1-\varepsilon}{|x|}\Big)^{\alpha/2} \Big(\frac{1+\varepsilon}{|x|}-r\Big)^{\alpha/2}
     }
     \dd r
      \notag \\
    &\hspace{4cm}+ \frac{2c_{\alpha, d}\pi^{d/2}}{\Gamma(d/2)}\int_{1}^{\frac{1+\varepsilon}{|x|}} 
     \frac{
      {_2}F_1 \Big(\frac{d-\alpha}{2}, 1-\frac{\alpha}{2}; \frac{d}{2}; r^{-2}\Big) r^{\alpha-1} 
     }{
      \Big(r-\frac{1-\varepsilon}{|x|}\Big)^{\alpha/2} \Big(\frac{1+\varepsilon}{|x|}-r\Big)^{\alpha/2}
     }
     \dd r.    \label{simplify}
   \end{align}
For the first term on the right-hand side of \eqref{simplify}, we can appeal to \eqref{bigF} and \eqref{f1} to deduce that 
\begin{align}
&\lim_{\varepsilon\to0}\sup_{x\in\mathbb{S}^{d-1}_\varepsilon}\Bigg|\frac{2c_{\alpha, d}\pi^{d/2}}{\Gamma(d/2)}\int_{\frac{1-\varepsilon}{|x|}}^{1} 
     \frac{
     {_2}F_1 \Big(\frac{d-\alpha}{2}, 1-\frac{\alpha}{2}; \frac{d}{2}; r^2\Big)  r^{d-1} 
     }{
      \Big(r-\frac{1-\varepsilon}{|x|}\Big)^{\alpha/2} \Big(\frac{1+\varepsilon}{|x|}-r\Big)^{\alpha/2}
     }
     \dd r\notag\\
     & \hspace{2cm} -
  \frac{2c_{\alpha, d}\pi^{d/2}\Gamma(1-\alpha)}{\Gamma((d-\alpha)/2) \Gamma((2-{\alpha})/{2})}  \int_{\frac{1-\varepsilon}{|x|}}^{1} 
     \frac{
     (1-r^2)^{\alpha-1} r^{d-1} 
     }{
      \Big(r-\frac{1-\varepsilon}{|x|}\Big)^{\alpha/2} \Big(\frac{1+\varepsilon}{|x|}-r\Big)^{\alpha/2}
     }
     \dd r\notag\\
&   \hspace{4cm}  -
     \frac{2c_{\alpha, d}\pi^{d/2}\Gamma(1-\alpha)}{\Gamma({\alpha}/{2})\Gamma((d+\alpha-2)/2)}   \int_{\frac{1-\varepsilon}{|x|}}^{1} 
     \frac{
     r^{d-1} 
     }{
      \Big(r-\frac{1-\varepsilon}{|x|}\Big)^{\alpha/2} \Big(\frac{1+\varepsilon}{|x|}-r\Big)^{\alpha/2}
     }\dd r\Bigg|=0.
     \label{supto0}
\end{align}
Note that, by using the transformation $r = (1-\varepsilon + 2\varepsilon u)/|x|$, 
\begin{align}
&\int_{\frac{1-\varepsilon}{|x|}}^{1} 
     r^{d-1} 
      \Big(r-\frac{1-\varepsilon}{|x|}\Big)^{-\alpha/2} \Big(\frac{1+\varepsilon}{|x|}-r\Big)^{-\alpha/2}
      \dd r\notag\\
      & \hspace{2cm}= |x|^{\alpha-d }(2\varepsilon)^{1-\alpha} \int_0^{(|x|-1+\varepsilon)/2\varepsilon} (2\varepsilon u+1-\varepsilon)^{d-1} u^{-\alpha/2}(1-u)^{-\alpha/2}\dd u\notag\\
      &\hspace{2cm}\leq |x|^{\alpha-d }(2\varepsilon)^{1-\alpha} \frac{\Gamma((2-\alpha)/2)^2}{\Gamma(2-\alpha)},
      \label{use1}
\end{align}
which tends to zero uniformly in $x\in\mathbb{S}_\varepsilon^{d-1}$ as $\varepsilon\to0$. 
\smallskip

The asymptotic \eqref{use1} also tells us that the approximating term of interest in \eqref{supto0} is  the middle term. For that, we can use \eqref{poly} to observe
\begin{align}
&\lim_{\varepsilon\to0}\sup_{x\in\mathsf{S}_\varepsilon}
\Bigg|
 \int_{\frac{1-\varepsilon}{|x|}}^{1} 
     (1-r^2)^{\alpha-1} r^{d-1} 
      \Big(r-\frac{1-\varepsilon}{|x|}\Big)^{-\alpha/2} \Big(\frac{1+\varepsilon}{|x|}-r\Big)^{-\alpha/2}
     \dd r\notag\\
     &\hspace{4cm}
     -2^{\alpha-1} \int_{\frac{1-\varepsilon}{|x|}}^{1} 
     (1-r)^{\alpha-1} 
      \Big(r-\frac{1-\varepsilon}{|x|}\Big)^{-\alpha/2} \Big(\frac{1+\varepsilon}{|x|}-r\Big)^{-\alpha/2}
     \dd r
     \Bigg|=0
 \label{preadd1}
\end{align}
and 
\begin{align}
&      \frac{2^\alpha c_{\alpha, d}\pi^{d/2}\Gamma(1-\alpha)}{\Gamma({\alpha}/{2})\Gamma((d+\alpha-2)/2)}  \int_{\frac{1-\varepsilon}{|x|}}^{1} 
     (1-r)^{\alpha-1} 
      \Big(r-\frac{1-\varepsilon}{|x|}\Big)^{-\alpha/2} \Big(\frac{1+\varepsilon}{|x|}-r\Big)^{-\alpha/2}
     \dd r\notag\\
     &=\frac{2^\alpha c_{\alpha, d}\pi^{d/2}\Gamma(1-\alpha)}{\Gamma({\alpha}/{2})\Gamma((d+\alpha-2)/2)}\int^{1-\frac{1-\varepsilon}{|x|}}_0 
     u^{\alpha-1} 
      \Big(1-\frac{1-\varepsilon}{|x|}-u\Big)^{-\alpha/2} \Big(\frac{1+\varepsilon}{|x|}-1+u\Big)^{-\alpha/2}
     \dd r\notag\\
     &=\frac{2^\alpha c_{\alpha, d}\pi^{d/2}\Gamma(1-\alpha) \Gamma((2-\alpha)/2)\Gamma(\alpha)}
     {\Gamma({\alpha}/{2})\Gamma((d+\alpha-2)/2)\Gamma((2+\alpha)/2)}
     \left(\frac{|x|-1+\varepsilon}{1+\varepsilon-|x|}\right)^{\alpha/2 }
\notag\\
&\hspace{6cm}         {_2}F_1 \Big(\alpha/2,\alpha;1+\alpha/2; -\frac{|x|-1+\varepsilon}{1+\varepsilon-|x|}\Big).
\label{add1}
\end{align}

The second term on the right-hand side of \eqref{simplify} can be dealt with similarly. Indeed, using \eqref{f1} we can  produce an analogous statement to \eqref{supto0}, from which, the leading order approximating term is the integral 
\begin{align}
& \frac{2c_{\alpha, d}\pi^{d/2}\Gamma(1-\alpha)}{\Gamma({\alpha}/{2})\Gamma((d+\alpha-2)/2)}  \int_{1}^{\frac{1+\varepsilon}{|x|}} 
     (1-r^{-2})^{\alpha-1} r^{d-1} 
      \Big(r-\frac{1-\varepsilon}{|x|}\Big)^{-\alpha/2} \Big(\frac{1+\varepsilon}{|x|}-r\Big)^{-\alpha/2}
     \dd r\notag\\
     &\sim  \frac{2^\alpha c_{\alpha, d}\pi^{d/2}\Gamma(1-\alpha)}{\Gamma({\alpha}/{2})\Gamma((d+\alpha-2)/2)}
    \int_{1}^{\frac{1+\varepsilon}{|x|}} 
     (r-1)^{\alpha-1}
      \Big(r-\frac{1-\varepsilon}{|x|}\Big)^{-\alpha/2} \Big(\frac{1+\varepsilon}{|x|}-r\Big)^{-\alpha/2}
     \dd r\notag\\
     &= \frac{2^\alpha c_{\alpha, d}\pi^{d/2}\Gamma(1-\alpha)}{\Gamma({\alpha}/{2})\Gamma((d+\alpha-2)/2)}
    \int_{0}^{\frac{1+\varepsilon}{|x|} -1} 
     u^{\alpha-1}
      \Big(u+1-\frac{1-\varepsilon}{|x|}\Big)^{-\alpha/2} \Big(\frac{1+\varepsilon}{|x|}-1-u\Big)^{-\alpha/2}
     \dd u\notag\\
     &=\frac{2^\alpha c_{\alpha, d}\pi^{d/2}\Gamma(1-\alpha) \Gamma((2-\alpha)/2)\Gamma(\alpha)}
     {\Gamma({\alpha}/{2})\Gamma((d+\alpha-2)/2)\Gamma((2+\alpha)/2)}
\left(\frac{1+\varepsilon- |x|}{|x|-1+\varepsilon}\right)^{\alpha/2}
 \notag\\
&\hspace{6cm}      {_2}F_1 \Big(\alpha/2,\alpha;1+\alpha/2; -\frac{1+\varepsilon- |x|}{|x| -1+\varepsilon}\Big)
\label{add2}
\end{align}
uniformly for $x\in \mathbb{S}^{d-1}_\varepsilon$ as $\varepsilon\to0$, where we have again used \eqref{poly} to develop the right-hand side.

Somewhat remarkably, if we add together the right-hand side of \eqref{add1} and \eqref{add2}, using the identity in \eqref{linearcombo2}, we see that the sum is equal to 
\begin{equation}
 \frac{2^\alpha c_{\alpha, d}\pi^{d/2}\Gamma(1-\alpha) \Gamma((2-\alpha)/2)}
     {\Gamma((d+\alpha-2)/2)} =1
\label{nicesum}
\end{equation}
where the equality with unity follows from the choice of  $c_{\alpha, d}$  in the statement of the lemma.

\smallskip

Piecing together then uniform estimates above as well as the simplification of the two integrals \eqref{add1} and \eqref{add2} as well as the decay of the term \eqref{use1} in \eqref{supto0} and the analogous term when dealing with the second term on the right-hand side of \eqref{simplify}, the statement of the lemma follows.
 \end{proof}

Next we deal with the term $U\mu^{(2)}_{\varepsilon,\delta}$.

\begin{lemma}\label{L1}Recalling that $c_{\alpha, d}$ is the constant given in Lemma \ref{unif}, take $\delta(\varepsilon) = \varepsilon^{(1-\alpha)/2(d-\alpha)}$, then 
\[
\limsup_{\varepsilon\to0} \sup_{x\in\mathsf{S}_\varepsilon}\varepsilon^{(\alpha-1)/2} U\mu_{\varepsilon,\delta(\varepsilon)}^{(2)} (x)\le C_{\alpha,d},
\]
where
\[
C_{\alpha, d} = c_{\alpha, d} \frac{ 2^{2-\alpha}\pi^{(d-1)/2}\Gamma((2-\alpha)/2)^2}{ \Gamma(2-\alpha)\Gamma((d-1)/2)}. 
\]
In particular 
\[
\lim_{\varepsilon\to0} \sup_{x\in\mathsf{S}_\varepsilon} U\mu_{\varepsilon,\delta(\varepsilon)}^{(2)} (x)=0.
\]
\end{lemma}
\begin{proof}
Since $x \in \mathsf{S}_\varepsilon$ and $y \in \hat{\mathsf{S}}_{\varepsilon}^{\delta}$, i.e. $|x-y|>\delta$, we have,
\begin{align}
   \sup_{x\in\mathsf{S}_\varepsilon} U\mu_{\varepsilon,\delta}^{(2)} (x) &= \int_{ \hat{\mathsf{S}}_{\varepsilon}^{\delta}} 
   \frac{1}{|x-y|^{d-\alpha}}\mu_\varepsilon (\dd y) \notag\\
    &\le \frac{1}{\delta^{d-\alpha}} \int_{\hat{\mathsf{S}}_{\varepsilon}^{\delta}} \mu_\varepsilon (\dd y)  \notag \\
    &\leq  \frac{1}{\delta^{d-\alpha}} \frac{2\pi^{(d-1)/2}}{\Gamma((d-1)/2)} \int_{1-\varepsilon}^{1+\varepsilon} r^{d-1}m_\varepsilon (r)  \dd r 
    \label{use3}
\end{align}
where $m_\varepsilon(r) = c_{\alpha, d} (r-(1-\varepsilon))^{-\alpha/2} (1+\varepsilon-r)^{-\alpha/2} $.
It is easy to see that 
\begin{align}
\int_{1-\varepsilon}^{1+\varepsilon} m_\varepsilon (r)  \dd r 
&=c_{\alpha, d} \int_{1-\varepsilon}^{1+\varepsilon} (r-(1-\varepsilon))^{-\alpha/2} (1+\varepsilon-r)^{-\alpha/2} \dd r\notag\\
&=c_{\alpha, d} \varepsilon^{1-\alpha}2^{1-\alpha}\frac{ \Gamma((2-\alpha)/2)^2}{ \Gamma(2-\alpha)} .
\label{2ndthing}
\end{align}
Putting \eqref{use3} and \eqref{2ndthing} we have
\begin{align}
   \sup_{x\in\mathsf{S}_\varepsilon} U\mu_{\varepsilon,\delta}^{(2)} (x) &\leq 
  c_{\alpha, d}  \frac{ 2^{2-\alpha}\pi^{(d-1)/2}\Gamma((2-\alpha)/2)^2}{ \Gamma(2-\alpha)\Gamma((d-1)/2)} 
  \times \frac{\varepsilon^{1-\alpha}}{\delta^{d-\alpha}}.
    \label{use4}
\end{align}
By choosing $\delta = \delta(\varepsilon)$, the result follows. 
\end{proof}

Let us now return to the proof of Theorem \ref{main2}. We show that we can make careful sense of  \eqref{o(1)} and \eqref{o(1)2}. 
 Using 
  \eqref{differencee} in \eqref{preuse} we, for $x\not\in\mathsf{S}$,
\begin{align}
U\mu^{\mathsf{S}}_\varepsilon(x) &
= \mathbb{E}_x\left[(U\mu^{(1)}_\varepsilon (X_{\tau_{\mathsf{S}_\varepsilon}}) -1) ; \tau_{\mathsf{S}_\varepsilon}<\infty\right] +  \mathbb{P}_x(\tau_{\mathsf{S}_\varepsilon}<\infty) -  \mathbb{E}_x\left[U\mu^{(2)}_\varepsilon (X_{\tau_{\mathsf{S}_\varepsilon}})  ; \tau_{\mathsf{S}_\varepsilon}<\infty\right] \notag \\
&\leq \mathbb{E}_x\left[(U\mu^{(1)}_\varepsilon (X_{\tau_{\mathsf{S}_\varepsilon}}) -1) ; \tau_{\mathsf{S}_\varepsilon}<\infty\right] +  \mathbb{P}_x(\tau_{\mathsf{S}_\varepsilon}<\infty). 
\label{bits}
\end{align}
Then, due to Lemmas 
\ref{unif}, for each $x\not\in \mathsf{S}$ and $\upsilon>0$, we can choose $\varepsilon$ sufficiently small such that 
\begin{align}
U\mu^{\mathsf{S}}_\varepsilon(x) \leq (1+\upsilon)\mathbb{P}_x(\tau_{\mathsf{S}_\varepsilon}<\infty) .
\label{lowerbound}
\end{align}
Since we can take $\upsilon$ arbitrarily small, we have the lower bound on a liminf version of the statement of Theorem \ref{main2}  given by 
\begin{align}
\liminf_{\varepsilon\to 0}\varepsilon^{\alpha-1}U\mu^{\mathsf{S}}_\varepsilon(x) &\le  \liminf_{\varepsilon \to 0}  \varepsilon^{\alpha-1}\mathbb{P}_x(\tau_{\mathsf{S}_\varepsilon}<\infty), \qquad x\not\in\mathsf{S}.
\label{liminf}
\end{align}

On the other hand, suppose instead of $\mathsf{S}$, we replace its role by $\mathsf{S}^{\delta(\varepsilon)}$, where $\delta(\varepsilon)$ was given in the statement of Lemma \ref{L1}, we have from the excessive property \eqref{excessive} associated to $U\mu^{\mathsf{S}^{\delta(\varepsilon)}}_\varepsilon$ that
\begin{align}
U\mu^{\mathsf{S}^{\delta(\varepsilon)}}_\varepsilon(x)
 &\geq \mathbb{E}_x\left[U\mu^{\mathsf{S}^{\delta(\varepsilon)}}_\varepsilon (X_{\tau_{\mathsf{S}_{\varepsilon}}})  ; \tau_{\mathsf{S}_\varepsilon}<\infty\right]
 \label{excessiveprop}
\end{align}
where we can choose $\varepsilon$ sufficiently small that the identity holds for all $x\not\in\mathsf{S}$. Now appealing to \eqref{differencee2}, we get 
\begin{align*}
U\mu^{\mathsf{S}^{\delta(\varepsilon)}}_\varepsilon(x) 
 &\geq\mathbb{E}_x\left[U\mu^{(1)}_\varepsilon (X_{\tau_{\mathsf{S}_{\varepsilon}}}) -1 ; \tau_{\mathsf{S}_\varepsilon}<\infty\right]+ \mathbb{P}_x(\tau_{\mathsf{S}_\varepsilon}<\infty)-\mathbb{E}_x\left[U\mu^{(2)}_{\varepsilon,\delta(\varepsilon)} (X_{\tau_{\mathsf{S}_\varepsilon}})  ; \tau_{\mathsf{S}_\varepsilon}<\infty\right] .
\end{align*}
Appealing to Lemmas \ref{L1} and \ref{unif}, for each $\upsilon>0$, we can choose $\varepsilon$ small enough such that, for each $x\not\in\mathsf{S}$,
\begin{align}
U\mu^{\mathsf{S}^{\delta(\varepsilon)}}_\varepsilon(x) &\geq 
 (1-\upsilon )\mathbb{P}_x(\tau_{\mathsf{S}_\varepsilon}<\infty), 
 \label{upperbound}
\end{align}
Hence, since we can choose $\upsilon$ as small as we like, we have 
\begin{equation}
\limsup_{\varepsilon\to 0}\varepsilon^{\alpha-1} U\mu^{\mathsf{S}^{\delta(\varepsilon)}}_\varepsilon(x) 
\geq  \limsup_{\varepsilon \to 0}  \varepsilon^{\alpha-1}\mathbb{P}_x(\tau_{\mathsf{S}_\varepsilon}<\infty), \qquad x\not\in\mathsf{S}.
\label{limsup}
\end{equation}

It follows from \eqref{liminf} and \eqref{limsup} that, as soon as 
\begin{equation}
\limsup_{\varepsilon\to 0}\varepsilon^{\alpha-1} U\mu^{\mathsf{S}^{\delta(\varepsilon)}}_\varepsilon(x) 
=\liminf_{\varepsilon\to 0}\varepsilon^{\alpha-1}U\mu^{\mathsf{S}}_\varepsilon(x),\qquad x\not\in\mathsf{S},
\label{limit}
\end{equation}
Noting that $U\mu^\mathsf{S}_\varepsilon \leq U\mu^{\mathsf{S}^{\delta(\varepsilon)}}_\varepsilon$, we have 
\[
\lim_{\varepsilon\to 0}\varepsilon^{\alpha-1}U\mu^{\mathsf{S}}_\varepsilon(x) 
= \lim_{\varepsilon \to 0}  \varepsilon^{\alpha-1}\mathbb{P}_x(\tau_{\mathsf{S}_\varepsilon}<\infty), \qquad x\not\in\mathsf{S}.
\]
Let us thus complete the proof by verifying the limit on the equality  \eqref{limit} holds. 

\smallskip

To this end, using that  $|x-y|^{\alpha-d}$ is continuous on $\mathsf{S}_\varepsilon$ and, when $x\not\in\mathsf{S}$, without less of generality, we can take $\varepsilon$ small enough so that  $x \notin \mathsf{S}_\varepsilon$.
For each $x\notin \mathsf{S}$, using the Mean Valued Theorem, there exists a  $r^*_\varepsilon \in (1-\varepsilon,1+\varepsilon)$ such that 
\begin{align}
U\mu^{\mathsf{S}}_\varepsilon(x) 
&=  \int_{\mathsf{S}_\varepsilon} |x-y|^{\alpha-d} m_\varepsilon(|y|)\ell_d(\dd y) \notag \\
&= (r_\varepsilon^*)^{d-1}\int_{\mathsf{S}} |x-r_\varepsilon^*\theta|^{\alpha-d} \sigma_1(\dd\theta)\int_{1-\varepsilon}^{1+\varepsilon}m_\varepsilon(r)\dd r
, 
\label{usefora=1}
\end{align}
where we recall that $m_\varepsilon(r) = c_{\alpha, d} (r-(1-\varepsilon))^{-\alpha/2} (1+\varepsilon-r)^{-\alpha/2}$. By using  \eqref{2ndthing} we get 
\begin{align}
   \varepsilon^{\alpha -1}U\mu^{\mathsf{S}}_\varepsilon(x) 
 = (r_\varepsilon^*)^{d-1} 2^{1-\alpha}c_{\alpha, d}\frac{ \Gamma((2-\alpha)/2)^2}{ \Gamma(2-\alpha)}
   \int_{\mathsf{S}} |x-r_\varepsilon^*\theta|^{\alpha-d} \sigma_1(\dd\theta), \qquad x\notin\mathsf{S}.
   \label{rstar}
\end{align}

Taking limits in \eqref{rstar} as $\varepsilon\to0$ and recalling the value of $c_{\alpha, d}$ from the statement of Lemma \ref{unif}, we have, for $x\not\in \mathsf{S}$
\begin{align}
\lim_{\varepsilon\to0}\varepsilon^{\alpha -1}U\mu^{\mathsf{S}}_\varepsilon(x)
&=2^{1-2\alpha}\frac{\Gamma((d+\alpha-2)/2)}{ \pi^{d/2}\Gamma(1-\alpha) }\frac{ \Gamma((2-\alpha)/2)}{ \Gamma(2-\alpha)}
   \int_{\mathsf{S}} |x-\theta|^{\alpha-d} \sigma_1(\dd\theta).
   \label{almost}
\end{align}
An application of the recursion formula for gamma functions allows us to identify the right-hand side as equal to that of 
the right-hand side of \eqref{gammarecursion}. 
Very little changes in the above calculation if we replace $\mathsf{S}$ by $\mathsf{S}^{\delta(\varepsilon)}$. 
As such, \eqref{almost} allows us to conclude  \eqref{limit}, and thus gives the statement of the Theorem \ref{main2}.
\hfill$\square$

\section{Proof of Theorem \ref{main2} (ii)}

The proof needs some adaptation when we deal with the case $\alpha = 1$.  Principally, we need to focus on Lemmas \ref{unif} and \ref{L1}. What is different in these two lemmas is that the normalisation constant $c_{\alpha,d}$ must now depend on $\varepsilon$. The replacement for Lemma \ref{unif} and Lemma \ref{L1} (combined into one result) now takes the following form.

\begin{lemma}
\label{L3}
Suppose that we define, for $0<\varepsilon<1$,
\begin{equation}
\mu_\varepsilon (\dd y) =\frac{ c_{1, d}}{|\log \varepsilon|} (|y|-(1-\varepsilon))^{-\alpha/2} (1+\varepsilon-|y|)^{-\alpha/2} \ell_d(\dd y),
\label{mueps2}
 \end{equation}
 and
\[
c_{1, d} =  \frac{\Gamma((d-1)/2)}{ \pi^{(d+1)/2}}.
\] 
\begin{itemize}
\item[(i)]We have 
\[
\lim_{\varepsilon\to0}\sup_{x\in\mathbb{S}^{d-1}_\varepsilon}|U\mu^{(1)}_\varepsilon(x) -1|=0.
\]
\item[(ii)] take $\delta(\varepsilon) =   |\log \varepsilon|^{-1/2(d-1)}$, then 
\[
\limsup_{\varepsilon\to0} \sup_{x\in\mathsf{S}_\varepsilon}\sqrt{|\log \varepsilon|} U\mu_{\varepsilon,\delta(\varepsilon)}^{(2)} (x) <\infty,
\]
so that 
\[
\lim_{\varepsilon\to0} \sup_{x\in\mathsf{S}_\varepsilon} U\mu_{\varepsilon,\delta(\varepsilon)}^{(2)} (x)=0.
\]

\end{itemize}
\end{lemma}
\begin{proof}
We give only a sketch proof of both parts for the interested reader to use as a guide to reproduce the finer details. 
\medskip

(i) The essence of the proof is an adaptation of the proof of Lemma \ref{unif}. We pick up the proof of the latter at the analogue of \eqref{simplify}, albeit $\alpha = 1$ and $c_{\alpha,d}$ is replaced by $c_{1,d}/|\log\varepsilon|$, i.e.
 \begin{align}
U\mu_\varepsilon^{(1)} (x) &=  \frac{2c_{1, d}\pi^{d/2}}{|\log\varepsilon|\Gamma(d/2)} \int_{\frac{1-\varepsilon}{|x|}}^{1} 
     \frac{
     {_2}F_1 \Big(\frac{d-1}{2}, \frac{1}{2}; \frac{d}{2}; r^2\Big)  r^{d-1} 
     }{
      \Big(r-\frac{1-\varepsilon}{|x|}\Big)^{1/2} \Big(\frac{1+\varepsilon}{|x|}-r\Big)^{1/2}
     }
     \dd r
      \notag \\
    &\hspace{4cm}+ \frac{2c_{1, d}\pi^{d/2}}{|\log\varepsilon|\Gamma(d/2)}\int_{1}^{\frac{1+\varepsilon}{|x|}} 
     \frac{
      {_2}F_1 \Big(\frac{d-1}{2}, \frac{1}{2}; \frac{d}{2}; r^{-2}\Big) r^{1-1} 
     }{
      \Big(r-\frac{1-\varepsilon}{|x|}\Big)^{1/2} \Big(\frac{1+\varepsilon}{|x|}-r\Big)^{1/2}
     }
     \dd r.    \label{simplify1}
   \end{align}
    Appealing to \eqref{f12}, noting that $\log (1-r^2)\sim \log(1-r) + \log 2$, as $r\to 1$, we can deduce that there is an unimportant constant, say $\chi$, such that
\begin{align}
&\lim_{\varepsilon\to0}\sup_{x\in\mathbb{S}^{d-1}_\varepsilon}\Bigg|\frac{2 c_{1, d}\pi^{d/2}}{|\log\varepsilon|\Gamma(d/2)}\int_{\frac{1-\varepsilon}{|x|}}^{1} 
     \frac{
     {_2}F_1 \Big(\frac{d-1}{2}, \frac{1}{2}; \frac{d}{2}; r^2\Big)  r^{d-1} 
     }{
      \Big(r-\frac{1-\varepsilon}{|x|}\Big)^{1/2} \Big(\frac{1+\varepsilon}{|x|}-r\Big)^{1/2}
     }
     \dd r\notag\\
     & \hspace{2cm} +
  \frac{2 c_{1, d}\pi^{d/2}}{|\log\varepsilon|\Gamma((d-1)/2) \Gamma(1/{2})}  \int_{\frac{1-\varepsilon}{|x|}}^{1} 
     \frac{
     r^{d-1}  \log(1-r)
     }{
      \Big(r-\frac{1-\varepsilon}{|x|}\Big)^{1/2} \Big(\frac{1+\varepsilon}{|x|}-r\Big)^{1/2}
     }
     \dd r\notag\\
&   \hspace{4cm}  -
    \frac{c_{1,d} \chi }{|\log\varepsilon|} \int_{\frac{1-\varepsilon}{|x|}}^{1} 
     \frac{
     r^{d-1} 
     }{
      \Big(r-\frac{1-\varepsilon}{|x|}\Big)^{1/2} \Big(\frac{1+\varepsilon}{|x|}-r\Big)^{1/2}
     }\dd r\Bigg|=0.
     \label{supto02}
\end{align}
A similar uniform limiting control can be undertaken by subtracting off analogous terms from the second integral in \eqref{simplify1}, i.e. the integral
\[
\frac{2c_{1, d}\pi^{d/2}}{|\log \varepsilon|\Gamma(d/2)}\int_{1}^{\frac{1+\varepsilon}{|x|}} 
     \frac{
      {_2}F_1 \Big(\frac{d-1}{2}, \frac{1}{2}; \frac{d}{2}; r^{-2}\Big)
     }{
      \Big(r-\frac{1-\varepsilon}{|x|}\Big)^{1/2} \Big(\frac{1+\varepsilon}{|x|}-r\Big)^{1/2}
     }
     \dd r.  
\]
Using \eqref{use1}, again noting $\alpha =1$, we can uniformly control the last term in \eqref{supto02} and note that it is $O(1/|\log\varepsilon|)$. Similarly to \eqref{preadd1}, the second term term in \eqref{supto02} has the same behaviour as 
\begin{equation}
- \frac{2 c_{1, d}\pi^{d/2}}{|\log\varepsilon|\Gamma((d-1)/2) \Gamma(1/{2})}   \int_0^{1-\frac{1-\varepsilon}{|x|}} \frac{ \log u}{(\frac{1+\varepsilon-|x|}{|x|}+u)^{1/2}(\frac{|x|-(1-\varepsilon)}{|x|}-u)^{1/2}}     \dd u. 
\label{logint}
\end{equation}
To evaluate \eqref{logint}, using the change of variable $ u = a - ({a+b})/({t^2+1})$
\begin{eqnarray}
    \int_0^a \frac{\log u 
    }{\sqrt{(b+u)(a-u)}}  \dd u&=& 2 \int_{\sqrt{\frac{b}{a}}}^\infty \log \Big(a-\frac{a+b}{t^2+1}\Big) \frac{\dd t}{t^2+1} \notag \\
    &=& \int^{\arctan \sqrt \frac{a}{b}}_{0} \log(a-(a+b)\sin^2 {w}) \dd{w} \notag \\
    &=& \int^{\arctan \sqrt \frac{a}{b}}_{0} \log a + \log\Big(1- \frac{\sin^2 {w}}{\frac{a}{a+b}}\Big) \dd{w} \notag \\
    &=& \arctan \sqrt \frac{a}{b}  \log (a+b) - L\Big(\frac{\pi}{2}-2\arctan \sqrt{\frac{a}{b}}\Big)  - \frac{\pi}{2} \log 2
    \label{lob}
\end{eqnarray}
where we have used formula 4.226(5) of \cite{table}, which tells us that 
\begin{equation}
    \int_0^u \log\Big(1-\frac{\sin^2 {w}}{\sin^2 {v}}\Big) \dd{w} = -u \log \sin^2{v} - L(\frac{\pi}{2}-{v}+u) - L(\frac{\pi}{2}-{v}-u)
\end{equation}
for any $-\frac{\pi}{2} \le u \le \frac{\pi}{2}$ and $|\sin u| \le |\sin {v}|$ where $L(x)$ is the Lobachevsky function. Note that, Lobachevsky's function is defined and represented as 
\begin{eqnarray}
    L(x) &=& -\int_0^x \log \cos \theta \,\dd \theta = x \log 2 - \frac{1}{2} \sum_{k=1}^\infty (-1)^{k-1} \frac{\sin 2kx}{k^2}. 
\end{eqnarray}
Using \eqref{lob} to evaluate \eqref{logint} as well to evaluate the partner integral to \eqref{logint}, which comes from the analogous control of the second integral in \eqref{simplify1}, we get a nice cancellation of terms (as happened at this stage of the argument for $\alpha \in(0,1)$), to give us the controlled feature that
\[
\lim_{\varepsilon\to0}\sup_{x\in\mathbb{S}^{d-1}_\varepsilon}
\left|
U\mu_\varepsilon^{(1)} (x) +   \frac{2 c_{1, d}\pi^{d/2}}{|\log\varepsilon|\Gamma((d-1)/2) \Gamma(1/{2})} \frac{\pi}{2}\log\varepsilon \right| = 0.
\]
Noting that with the indicated choice of $c_{1,d}$, we have 
\[
 \frac{2 c_{1, d}\pi^{d/2}}{\Gamma((d-1)/2) \Gamma(1/{2})} \frac{\pi}{2} = 1,
 \]
 which concludes the proof of part (i).

 \medskip
 
 (ii) For the second part, the proof is almost identical to the proof of Lemma \ref{L1}. Indeed, following the calculations through to \eqref{use4}, recalling that we have replaced $c_{\alpha, d}$ by $c_{1,d}/|\log\varepsilon|$, we get, up to an unimportant constant $\chi'$, 
 \begin{align}
   \sup_{x\in\mathsf{S}_\varepsilon} U\mu_{\varepsilon,\delta}^{(2)} (x) &\leq \chi' \frac{1}{|\log\varepsilon|\delta^{d-1}}.
    \label{usea=1}
\end{align}
Hence, by taking $\delta = \delta (\varepsilon) = |\log \varepsilon|^{-1/2(d-1)}$ the statement of part (ii) follows. 
\end{proof}

With Lemma \ref{L3} in hand, we can now complete the proof of Theorem \ref{main2} (ii). Inequalities \eqref{lowerbound} and \eqref{upperbound} are still at our disposal for the same reasons as before. The proof thus boils down to the 
asymptotic  treatment  of the term $U\mu^\mathsf{S}_\varepsilon(x)$ as in \eqref{usefora=1} for $x\not\in\mathsf{S}$.
Recalling that we have replaced $c_{\alpha,d}$ by $c_{1,d}/|\log\varepsilon|$ we get from \eqref{2ndthing}
and the constant $c_{1,d}$ given in the statement of Lemma \ref{L3},
\[
\lim_{\varepsilon\to0}\ |\log\varepsilon|\  \mathbb{P}_x(\tau_{\mathsf{S}_\varepsilon}<\infty) = \frac{\Gamma((d-1)/2)}{ \pi^{(d+1)/2}}\Gamma(1/2)^2
H_\mathsf{S}(x)=\frac{\Gamma((d-1)/2)}{ \pi^{(d-1)/2}}H_\mathsf{S}(x)
\]
where we have used that $\Gamma(1/2) = \sqrt{\pi}$.
\qed

\section{Proof of Theorem \ref{main}}
Recall the definition $\tau_\beta : = \inf\{t>0 \colon 1/{\beta} < |X_t| < \beta\}$ for  $\beta > 1.$ Then, by Markov property applied at time $\tau_\beta$, we have, for any $\Lambda \in \mathcal{F}_t,$   
\begin{equation}
    \mathbb{P}_x (\Lambda, t < \tau_\beta | \tau_{\mathsf{S}_\varepsilon} < \infty) = \mathbb{E}_x \left[ \mathbf{1}_{\{\Lambda,t < {\tau_\beta} \}} \frac{\mathbb{P}_{X_t} (\tau_{\mathsf{S}_\varepsilon}  < \infty)}{\mathbb{P}_x (\tau_{\mathsf{S}_\varepsilon}  < \infty)} \right].
    \label{expect}
\end{equation}
The event $\{t<\tau_\beta\}$ implies that either  $|X_t|>\beta>1$ or $|X_t|<{1}/{\beta}<1$. Moreover, there exists a $\varepsilon_0>0$ such that, for all $0<\varepsilon<\varepsilon_0$,  $(1-\varepsilon, 1+\varepsilon)\subset (1/\beta, \beta)$. Hence, 
for all $0<\varepsilon<\varepsilon_0$ and $y\in \mathbb{S}^{d-1}_\varepsilon$,  on $\{t<\tau_\beta\}$,
\[
|X_t-y|^{\alpha - d}<\max\{((1-\varepsilon_0)-1/\beta)^{\alpha-d}, (\beta -(1+\varepsilon_0))^{\alpha-d}\},
\]
Hence, on $\{t<\tau_\beta\}$, we have from \eqref{upperbound} and \eqref{rstar} that we can choose $\varepsilon$ sufficiently small such that 
\[
\varepsilon^{\alpha -1}\mathbb{P}_{X_t} (\tau_{\mathsf{S}_\varepsilon}  < \infty)< K_1
\]
for some constant $K_1\in(0,\infty)$. In a similar spirit, using \eqref{lowerbound} and \eqref{rstar}, since $x\not\in \mathsf{S}$ and $\mathsf{S}$ is closed, 
it follows similarly that there is another constant $K_2\in(0,\infty)$ such that, for $x$ given in \eqref{expect}, we can choose $\varepsilon$ sufficiently small such that 
\[
{\varepsilon^{\alpha -1}\mathbb{P}_{x} (\tau_{\mathsf{S}_\varepsilon}  < \infty)}>K_2.
\]

Theorem \ref{main2}, dominated convergence and    monotone convergence gives us, for all $\Lambda\in \mathcal{F}_t$, $t\geq0$,
\begin{align*}
\lim_{\beta\to1} \lim_{\varepsilon\to0}   \mathbb{P}_x (\Lambda, t < \tau_\beta | \tau_{\mathsf{S}_\varepsilon} < \infty)& = \lim_{\beta\to1}\mathbb{E}_x \left[ \mathbf{1}_{\{\Lambda,t < {\tau_\beta} \}}  \lim_{\varepsilon\to0} \frac{\varepsilon^{\alpha-1}\mathbb{P}_{X_t} (\tau_{\mathsf{S}_\varepsilon}  < \infty)}{\varepsilon^{\alpha-1}\mathbb{P}_x (\tau_{\mathsf{S}_\varepsilon}  < \infty)} \right]=\mathbb{E}_x \left[ \mathbf{1}_{\Lambda}  \frac{H_\mathsf{S}(X_t)}{H_\mathsf{S}(x)} \right],
\end{align*}
as required.\hfill$\square$

\section{Proof of Theorem \ref{Naga}}

Recall the notation for a general Markov process $(Y,\texttt{P})$ on $E$ preceding the statement of Theorem \ref{Naga}. 
We will additionally write $\mathcal{P}: = (\mathcal{P}_t, t\geq 0)$ for the   semigroup associated to $(Y,\texttt{P})$.

\smallskip

 Theorem 3.5 of Nagasawa \cite{Naga}, shows that, under suitable assumptions on the Markov process, $L$-times form a natural family of random times  at which the pathwise time-reversal 
\[
\stackrel{_\leftarrow}{Y}_{\!t}:=Y_{(\texttt{k}-t)-},\qquad  t\in [0,\texttt{k}],
\]
 is again a Markov process.  Let us state Nagasawa's principle assumptions.

\smallskip

\textbf{(A)} The potential measure $U_Y(a, \cdot)$ associated to $\mathcal{P}$, defined by the relation 
\begin{equation}
\int_Ef(x)U_Y(a,\dd x) = \int_0^\infty \mathcal{P}_t[f](a)\dd t={\texttt E}_a\left[\int_0^\infty f(X_t)\,\dd t\right],\qquad a\in E,
\label{GY}
\end{equation}
 for bounded and measurable $f$ on $E$, is $\sigma$-finite. Assume  that there exists  a probability measure, $\nu$, such that, if we put
\begin{align}\label{a1}
	\mu(A)=\int U_Y(a,A)\, \nu(\dd a)\quad \text{ for }A\in \mathcal B(\R),
\end{align}
	then there exists a Markov transition semigroup, say $\hat{\mathcal{P}}: = (\hat{\mathcal P}_t, t\geq 0)$ such that 	\begin{align}
		\int_E \mathcal{P}_t[f](x) g(x)\, \mu(\dd x)=\int_E f(x) \hat{\mathcal P}_t [g](x)\,\mu(\dd x),\quad t\geq 0,
		\label{weakdualtity}
	\end{align}
	for bounded, measurable and compactly supported test-functions $f, g$ on $E$.\smallskip
	


\smallskip

\textbf{(B)} For any continuous test-function $f\in C_0(E)$, the space of continuous and compactly supported functions,  and $a\in E$, assume that $\mathcal{P}_t[f](a)$ is right-continuous in $t$ for all $a\in E$ and, for $q> 0$, ${U}_{\hat Y}^{(q)}[f](\stackrel{_\leftarrow}{Y}_{\!t})$ is right-continuous in $t$, where, for bounded and measurable $f$ on $E$,
 \[
{U}_{\hat Y}^{(q)}[f](a) =\int_0^\infty  {\rm e}^{-qt}\hat{\mathcal{P}}_t[f](a){\dd t},\qquad  a\in E\]
is the $q$-potential associated to $\hat{\mathcal P}$.
\smallskip

 Nagasawa's  duality theorem, Theorem 3.5. of \cite{Naga}, now reads as follows.

 \begin{theorem}[Nagasawa's duality theorem]\label{Ndual} Suppose that assumptions {\rm{\bf (A)} } and {\rm{\bf (B)}} hold. For the given starting probability distribution $\nu$ in {\rm{\bf (A)} } and any $L$-time $\emph{\texttt{k}}$, the time-reversed process $\stackrel{_\leftarrow}{Y}$ under $\emph{\texttt P}_\nu$ is a time-homogeneous Markov process with transition probabilities
\begin{align}
	\emph{\texttt{P}}_\nu(\stackrel{_\leftarrow}{Y}_t \in A\,|\stackrel{_\leftarrow}{Y}_r, 0<r< s)=\emph{\texttt{P}}_\nu(\stackrel{_\leftarrow}{Y}_t \in A\,|\stackrel{_\leftarrow}{Y}_s)={p}_{\hat{Y}}(t-s,\stackrel{_\leftarrow}{Y}_s,A),\quad \emph{\texttt{P}}_\nu\text{-almost surely},
\end{align}
for all $0<s<t$ and closed $A$ in $\mathbb{R}$, where ${p}_{\hat{Y}}(u, x, A)$, $u\geq 0$, $x\in\mathbb{R}$, is the transition measure associated to the semigroup $\hat{\mathcal P}$.
\end{theorem}

\subsubsection*{Completing the proof of Theorem \ref{Naga}}
We will make a direct application of Theorem \ref{Ndual}, with $Y$ taken to be the process $(X,\mathbb{P}_\nu)$ where $\nu$ satisfies \eqref{closed}.  Recall that its potential is written $U$ and we will denote its transition semigroup by $(\mathcal{P}_t, t\geq0)$.  Moreover, the dual process, formerly $\hat Y$,  is taken to be $(X,\mathbb{P}^\mathsf{S})$ and we will, in the obvious way, work with the notation $U^\mathsf{S}$ in place of $U_{\hat{Y}}$, $\mathcal{P}^\mathsf{S}$ in place of $\hat{\mathcal{P}}$ and so on. We need only to verify the two assumptions {\bf (A)} and {\bf (B)}.
\smallskip

In order to verify {\bf (A)}, writing
\begin{equation*}
     U(x,\dd y)=\int_0^\infty \mathbb{P}_{x} (X_t \in{\dd y}) \dd t = \frac{\Gamma((d-\alpha)/2)}{2^{\alpha} \pi^{d/2}\Gamma(\alpha/2)} |x-y|^{\alpha-d}\ell_d({\dd y}), \qquad x,y\in \mathbb{R}^d,
\end{equation*}

 we have, up to a multiplicative constant, 
\begin{eqnarray}
    \eta(\dd x) &=& \int_{\mathbb{R}^d} U(a,\dd x) \nu ({\dd a}) = \frac{1}{\sigma_1(\mathsf{S})} \int_{\mathsf{S}} |x-a|^{\alpha-d} \sigma_1 ({\dd a}) \propto H_{\mathsf{S}}(x)  \dd x.
    \label{eta}
\end{eqnarray}
Now, we need to verify that  \eqref{weakdualtity} holds. Hunt's switching identity (cf. Chapter II.1 of  \cite{Bertoin}) for  $(X, \mathbb{P})$,
states that 
\[
\mathcal{P}_t (y, \dd x)\dd y = \mathcal{P}_t (x, \dd y)\dd x, \qquad x,y \in\mathbb{R}^d.
\]
Using Hunt's switching identity together with \eqref{eta},  we have  for $x,y\in\mathbb{R}^d \setminus \mathsf{S}$
\begin{align*}
 \mathcal{P}_t (y, \dd x) \eta(\dd y) = \mathcal{P}_t (y,\dd x)  H_\mathsf{S}(y)  \dd y = \mathcal{P}_t (x,\dd y) \frac{H_\mathsf{S}(y)}{H_\mathsf{S}(x)} H_\mathsf{S}(x) \dd x =  \mathcal{P}_t^\mathsf{S} (x,\dd y) \eta(\dd x).
\end{align*}
Let us now turn to the verification of assumption {\bf (B)}.
This assumption is immediately satisfied on account of the fact that $\mathcal{P}^\mathsf{S}$ is a right-continuous semigroup by virtue of its definition as a Doob $h$-transform with respect to the Feller semigroup $\mathcal{P}$ of the stable process. 

\smallskip

With both {\bf (A)} and {\bf (B)} in hand, we are ready to apply Theorem \ref{Ndual} and the desired result thus follows.\hfill$\square$

\section{Proof of Theorem \ref{main5}}
For the proof of Theorem \ref{main5}, we focus on just part (i) and (ii) as the proof of parts (iii)--(v) are essentially verbatim the same as for the case of $\mathsf{S}\in \mathbb{S}^{d-1}$.  Moreover, for both parts (i) and (ii) we will provide only a sketch proof as the reader will quickly see that the proof is not hugely different form that of Theorem \ref{main2}, albeit for a few technical details.

\smallskip

 (i) The substance of the proof of part (i) is thus to follow a similar strategy as with Theorem \ref{main2} and build a measure $\rho^\mathsf{D}_\varepsilon$ such that the analogue of \eqref{o(1)} holds, i.e. $U\rho^\mathsf{D}_\varepsilon(x) = 1 + o(1)$, for $x\in \mathsf{D}$ so that 
$
(1+o(1))\mathbb{P}_x(\tau_{\mathsf{D}\varepsilon} <\infty) = U\rho^{\mathsf{D}}_\varepsilon (x)$, $x\not\in \mathsf{D}\varepsilon
$.
More precisely, we develop analogues of Lemmas  \ref{unif} and \ref{L1} to help make this precise. 

\smallskip

Following what we have learned for $\mu^\mathsf{S}_\varepsilon$, our choice of $\rho^\mathsf{D}_\varepsilon$ is built from the base  measure  
\begin{equation}
\rho_\varepsilon(\dd y)=k_{\alpha, d}((v,y)+\varepsilon)^{-\alpha/2}(\varepsilon-(v,y))^{-\alpha/2} \ell_d(\dd y).
\label{planarmueps}
\end{equation}
for an appropriate choice of $k_{\alpha, d}$.
As in \eqref{differencee} 
the we can work with the decomposition, 
\begin{equation}
U\rho^{\mathsf{D}}_\varepsilon(x) = U\rho^{(1)}_\varepsilon (x) - U\rho^{(2)}_\varepsilon(x) \qquad x\in \mathsf{D}_\varepsilon,
\label{differencee3}
\end{equation}
where $\rho^{(1)}_\varepsilon$ (resp. $\rho^{(2)}_\varepsilon$) is the restriction of $\rho_\varepsilon$ to $\mathbb{H}^{d-1}_\varepsilon: = \{x\in\mathbb{R}^d:  -\varepsilon< (v, x)<\varepsilon \}$  (resp. to $\hat{\mathsf{D}}_\varepsilon: = \mathbb{H}_\varepsilon^{d-1} \setminus \mathsf{D}_\varepsilon$).
This  helps with lower bounding $\liminf_{\varepsilon \to 0}  \varepsilon^{\alpha-1}\mathbb{P}_x(\tau_{\mathsf{D}_\varepsilon}<\infty)$ by following steps of \eqref{bits}--\eqref{liminf} together with the last paragraph of the Proof of Theorem \ref{main2}, for which an analogue of Lemma \ref{unif} is needed.

\smallskip

For each $|u|<\varepsilon$, define the following sets: 
 $\mathsf{D}^\delta = \{x\in\mathbb{H}^{d-1}: {\rm dist}(x, \mathsf{D})<\delta\}$, $\mathsf{D}^\delta_\varepsilon= \{y \in \mathbb{H}^{d-1}_\varepsilon \colon \hat{y}\in \mathsf{D}^\delta \}$ (recalling $\hat y$ is the orthogonal projection of $y$ on to $\mathbb{H}^{d-1}$)  and $\hat{\mathsf{D}}^\delta_\varepsilon = \mathbb{H}^{d-1}_\varepsilon \setminus \mathsf{D}^\delta_\varepsilon$.  Similarly, in the spirit of \eqref{differencee2} we can use the decomposition 
\begin{equation}
U\rho^{\mathsf{D}^\delta}_\varepsilon(x) = U\rho^{(1)}_\varepsilon (x)  - U\rho^{(2)}_{\varepsilon,\delta} (x) \qquad x\in \mathsf{D}_\varepsilon,
\label{differencee4}
\end{equation}
where $\rho^{(2)}_{\varepsilon,\delta}$ is the restriction of of $\rho_\varepsilon$ to $\hat{\mathsf{D}}_\varepsilon^\delta$.
which helps with $\limsup_{\varepsilon \to 0}  \varepsilon^{\alpha-1}\mathbb{P}_x(\tau_{\mathsf{D}_\varepsilon}<\infty)$ by following steps \eqref{excessiveprop}--\eqref{limit} together with the last paragraph of the Proof of Theorem \ref{main2},  for which an analogue of Lemma \ref{L1} is needed.

\smallskip

Let us address the technical detail that differs from the proof of Theorem \ref{main2} that we alluded to above. 
For $x\in\mathsf{D}_\varepsilon$,
\begin{align*}
    &U\rho_\varepsilon^{(1)} (x) \notag\\
    &= k_{\alpha,d} \int_{\mathbb{H}_\varepsilon^{d-1}} |x-y|^{\alpha-d} ((v,y)+\varepsilon)^{-\alpha/2} (\varepsilon-(v,y))^{-\alpha/2}\ell_d(\dd y) \notag \\
    &= k_{\alpha,d} \int_{-\varepsilon}^{\varepsilon}  (u+\varepsilon)^{-\alpha/2} (\varepsilon-u)^{-\alpha/2} \dd u \int_{\mathbb{H}^{d-1}(u)} |x-y|^{\alpha-d}\ell_{d-1}(\dd y) \notag \\
    &= k_{\alpha,d} \int_{-\varepsilon}^{\varepsilon}  (u+\varepsilon)^{-\alpha/2} (\varepsilon-u)^{-\alpha/2} \dd u \int_{\mathbb{H}^{d-1}((v,x))} (|x-\hat{y}|^2+|(v,x)-u|^2)^{\frac{\alpha-d}{2}} \ell_{d-1}(\dd\hat{y}) 
\end{align*}
where $\hat{y}$ is the orthogonal projection of $y\in \mathbb{H}^{d-1}(u)$ on to $\mathbb{H}^{d-1}(x)$, which   satisfies $|\hat{y}-y|={|(v,x)-u|}$ and $\ell_{d-1}(\dd\hat{y})=\ell_{d-1}(\dd y)$. Note also that $(v,x- \hat{y})=0$, for $\hat{y}\in\mathbb{H}^{d-1}((v, x))$, and hence  $x- \mathbb{H}^{d-1}((v, x))$ is equal to $\mathbb{H}^{d-1}(0)$, which, in turn, can otherwise be identified as $\mathbb{R}^{d-1}$,  we have
\begin{align}
    &U\rho_\varepsilon^{(1)} (x) \notag\\
    &= k_{\alpha,d}  \int_{-\varepsilon}^{\varepsilon}  (u+\varepsilon)^{-\alpha/2} (\varepsilon-u)^{-\alpha/2}\dd u \int_{\mathbb{H}^{d-1}((v,x))} \Big(|x-\hat{y}|^2+|(v,x)-u|^2\Big)^{\frac{\alpha-d}{2}} \ell_{d-1}(\dd\hat{y}) \notag \\
    &=\frac{2 k_{\alpha,d}  \pi^{(d-2)/2}}{\Gamma((d-2)/2)} \int_{-\varepsilon}^{\varepsilon}  (u+\varepsilon)^{-\alpha/2} (\varepsilon-u)^{-\alpha/2}\dd u 
    \int_0^\infty \int_{\mathbb{S}^{d-2}} \ \Big(r^2+{|(v,x)-u|^2}\Big)^{\frac{\alpha-d}{2}} r^{d-2}{\dd r}\sigma_1(d\theta) \notag \\
    &= \frac{2 k_{\alpha,d}  \pi^{(d-2)/2}}{\Gamma((d-2)/2)}  \int_{-\varepsilon}^{\varepsilon}  (u+\varepsilon)^{-\alpha/2} (\varepsilon-u)^{-\alpha/2}\dd u \int_0^\infty \Big(r^2+{|(v,x)-u|^2}\Big)^{\frac{\alpha-d}{2}} r^{d-2}{\dd r} \notag \\
    &= \frac{ k_{\alpha,d}  \pi^{(d-2)/2}}{\Gamma((d-2)/2)}  \int_{-\varepsilon}^{\varepsilon}  (u+\varepsilon)^{-\alpha/2} (\varepsilon-u)^{-\alpha/2}\dd u \int_0^\infty \Big(w+{|(v,x)-u|^2}\Big)^{\frac{\alpha-d}{2}} w^\frac{d-3}{2}{\dd w}\label{pickup} \\
    &= \frac{ k_{\alpha,d}  \pi^{(d-2)/2}\Gamma(\frac{1-\alpha}{2})\Gamma(\frac{d-1}{2})}
    {\Gamma(\frac{d-2}{2})\Gamma(\frac{d-\alpha}{2})} 
    \int_{-\varepsilon}^{\varepsilon}  (u+\varepsilon)^{-\alpha/2} (\varepsilon-u)^{-\alpha/2} {|(v,x)-u|^{\alpha-1}} \dd u \notag\\
    &=\frac{ k_{\alpha,d}  \pi^{(d-2)/2}\Gamma(\frac{1-\alpha}{2})\Gamma(\frac{d-1}{2})}
    {\Gamma(\frac{d-2}{2})\Gamma(\frac{d-\alpha}{2})} 
    \int_{-1}^{1}(1+w)^{-\alpha/2} (1-w)^{-\alpha/2} {|{\varepsilon^{-1}}{(v,x)}-w|^{\alpha-1}} \dd w
   \label{ux}
\end{align}
where, in the second equality, we have used generalised polar coordinates to integrate over   $\mathbb{H}^{d-1}(0)=\mathbb{R}^{d-1}$, in the penultimate equality, we used a classical representation of the Beta function (see formula 3.191.2 in \cite{table}), which tells us that, for any ${\rm Re}(\nu) > {\rm Re}(\gamma) >0$ and $z>0$, 
\begin{equation*}
    \int_0^\infty (y+z)^{-\nu} y^{\gamma-1} \dd y = z^{\gamma-\nu} \frac{\Gamma(\nu-\gamma)\Gamma(\gamma)}{\Gamma(\nu)},
\end{equation*}
and in the final equality, we have changed variables using $w = \varepsilon u$. 
Next, we observe that $|{\varepsilon^{-1}}{(v,x)}|\leq 1$ on account of the fact that $x\in\mathsf{D}_\varepsilon\subseteq \mathbb{H}^{d-1}_\varepsilon$.
Now choose $k_{\alpha, d}$, so that 
\begin{equation}
 \frac{ k_{\alpha,d}  \pi^{(d-2)/2}\Gamma(\frac{1-\alpha}{2})\Gamma(\frac{d-1}{2})}
    {\Gamma(\frac{d-2}{2})\Gamma(\frac{d-\alpha}{2})} =1.
    \label{kad}
\end{equation}
We can now appeal directly to \eqref{Simon} to deduce that, for $x\in \mathsf{D}_\varepsilon$ 
\begin{equation}
U\rho_\varepsilon^{(1)} (x) =1.
\end{equation}

In the spirit of \eqref{bits}--\eqref{liminf}, it now follows that, for $x\not\in \mathsf{D}$ and $\varepsilon$ sufficiently small, 
\[
U\rho^{\mathsf{D}}_\varepsilon(x) \leq \mathbb{P}_x(\tau_{\mathsf{D}_\varepsilon}<\infty) 
\]
so that 
\begin{equation}
\liminf_{\varepsilon\to0}U\rho^{\mathsf{D}}_\varepsilon(x) \leq \liminf_{\varepsilon\to0}\mathbb{P}_x(\tau_{\mathsf{D}_\varepsilon}<\infty), \qquad x\not\in\mathsf{D}.
\label{liminf2}
\end{equation}

Now we turn our attention to \eqref{differencee4}.   Define $\mathbb{H}^{d-1}(u) = \{x\in\mathbb{R}^d: (v, x) = u\}$, $\hat{\mathsf{D}}^\delta(u) = \mathbb{H}^{d-1}(u) \setminus \mathsf{D}^\delta(u)$ where $\mathsf{D}^\delta (u):= \{y \in \mathbb{H}^{d-1} (u)\colon \hat{y}\in \mathsf{D}^\delta \}$. Noting that when $x\in \mathsf{D}_\varepsilon$, $|x - y|>\delta$ for $y\in\hat{\mathsf{D}}^\delta_\varepsilon$, we have, for all $x\in \mathsf{D}_\varepsilon$,
\begin{align*}
U\rho^{(2)}_{\varepsilon,\delta} (x)&= k_{\alpha,d} \int_{\hat{\mathsf{D}}^\delta_\varepsilon} |x-y|^{\alpha-d} ((v,y)+\varepsilon)^{-\alpha/2} (\varepsilon-(v,y))^{-\alpha/2}\ell_d(\dd y) \notag \\ 
&\leq  k_{\alpha,d}\delta^{\alpha -d} \int_{-\varepsilon}^{\varepsilon}  (u+\varepsilon)^{-\alpha/2} (\varepsilon-u)^{-\alpha/2} \dd u
\int_{\mathsf{D}^{\delta}((v,x))}  \ell_{d-1}(\dd\hat{y}) \notag\\
&\leq \delta^{\alpha -d} k_{\alpha,d} \ell_{d-1}(\mathsf{D}^{\delta}) \int_{-\varepsilon}^{\varepsilon}  (u+\varepsilon)^{-\alpha/2} (\varepsilon-u)^{-\alpha/2} \dd u
 \notag\\
 &= \delta^{\alpha -d}\varepsilon^{1-\alpha}  k_{\alpha,d} \ell_{d-1}(\mathsf{D}^{\delta})2^{1-\alpha}\frac{ \Gamma((2-\alpha)/2)^2}{ \Gamma(2-\alpha)} ,
\end{align*}
where we have used the calculation in  \eqref{2ndthing} in the final equality. Choosing $\delta =\delta(\varepsilon) = \varepsilon^{(1-\alpha)/2(d-\alpha)}$, and noting that  $\ell_{d-1}(\mathsf{D}^{\delta})$ is uniformly bounded from above by an unimportant constant for e.g. all $\delta<1$ (thanks to the assumption that $\ell_{d-1}(\mathsf{D})<\infty$), we see that 
\[
\lim_{\varepsilon\to0}\sup_{x\in\mathsf{D}_\varepsilon} U\rho_{\varepsilon,\delta(\varepsilon)}^{(2)} (x)=0. 
\]

In a similar spirit to \eqref{excessiveprop}--\eqref{limsup}, we now have that 
\begin{equation}
\limsup_{\varepsilon\to 0}\varepsilon^{\alpha-1} U\rho^{\mathsf{D}^{\delta(\varepsilon)}}_\varepsilon(x) 
\geq  \limsup_{\varepsilon \to 0}  \varepsilon^{\alpha-1}\mathbb{P}_x(\tau_{\mathsf{D}_\varepsilon}<\infty), \qquad x\not\in\mathsf{D}.
\label{limsup2}
\end{equation}
Matching up the left-hand side of \eqref{liminf2} with that of \eqref{limsup2}, we can proceed in a similar fashion to \eqref{rstar} -- \eqref{almost}, leading to the statement of Theorem \ref{main5} (i) as promised.  The calculation is based around the fact that 
\begin{align}
\lim_{\varepsilon\to0}\varepsilon^{\alpha -1}U\rho^\mathsf{D}_\varepsilon(x)
&= \lim_{\varepsilon\to0} k_{\alpha, d} \varepsilon^{\alpha -1}\int_{\mathsf{D}_\varepsilon}|x-y|^{\alpha -d}((v,y)+\varepsilon)^{-\alpha/2}(\varepsilon-(v,y))^{-\alpha/2} \ell_d(\dd y)\notag\\
&=\lim_{\varepsilon\to0} k_{\alpha, d} \varepsilon^{\alpha -1 }\int_{-\varepsilon}^\varepsilon(u+\varepsilon)^{-\alpha/2}(\varepsilon-u)^{-\alpha/2} \dd u \int_{\mathsf{D}(u)}|x-\hat{y}|^{\alpha-d}\ell_{d-1}(\dd \hat y)\notag\\
&=k_{\alpha, d}  2^{1-\alpha}\frac{ \Gamma((2-\alpha)/2)^2}{ \Gamma(2-\alpha)}
\int_{\mathsf{D}}|x-{y}|^{\alpha-d}\ell_{d-1}(\dd  y)\notag\\
&= 2^{1-\alpha}\pi^{-(d-2)/2}\frac{\Gamma(\frac{d-2}{2})\Gamma(\frac{d-\alpha}{2}) \Gamma(\frac{2-\alpha}{2})^2}{\Gamma(\frac{1-\alpha}{2})\Gamma(\frac{d-1}{2})\Gamma(2-\alpha)}
\int_{\mathsf{D}}|x-{y}|^{\alpha-d}\ell_{d-1}(\dd y),
\label{somethingsimilarbelow}
 \end{align}
where    $\mathsf{D} (u):= \{y \in \mathbb{H}^{d-1} (u)\colon \hat{y}\in \mathsf{D} \}$ and we have used the calculation in \eqref{2ndthing} and \eqref{kad} in the third equality.

\medskip

(ii) The setting $\alpha = 1$ requires yet another delicate handing of the associated potentials. Given that all the main ideas are now present in the paper, we simply lay out the key points of the proof, leaving the remaining detail for the reader.
\smallskip

  Our calculations begin in the same way as in part (i), in particular, we work with the core measure $\rho_\varepsilon$ as in \eqref{planarmueps}, albeit (as with Theorem \ref{main2} (ii)) replacing $k_{1, d}$ by $k_{1,d}/|\log\varepsilon|$, to be used in the constructions \eqref{differencee3} and \eqref{differencee4}.  An immediate complication we have is in evaluating $U\rho_\varepsilon^{(1)}(x)$, for $x\in\mathsf{D}_\varepsilon$, can be seen when we pick up the computations for part (i) at \eqref{pickup}. Indeed, at that point, we are confronted with the integral
\[
 \int_0^\infty \Big(w+{|(v,x)-u|^2}\Big)^{\frac{1-d}{2}} w^\frac{d-3}{2}{\dd w} = \infty.
\]
The solution to this is to adjust the core measure $\rho_\varepsilon$ as follows. Since $\mathsf{D}$ is bounded, we can choose an $R>0$ sufficiently large that, for all $x\in\mathsf{D}$ in  $\mathbb{S}^{d-2}(0, R): = \{y\in\mathbb{H}^{d-1}: |y|\leq R\}$ strictly contains $\mathsf{D}$.  Denote $\mathbb{S}^{d-2}_\varepsilon(0, R) = \{x\in \mathbb{R}^d: |\hat{x}-x|\leq \varepsilon\}$, where $\hat{x}$ is the orthogonal projection of $x$ on to $\mathbb{H}^{d-1}$.  Suppose we now make a slight adjustment and replace $\rho_\varepsilon$ by 
\[
\rho_\varepsilon (\dd y) = \frac{k_{1, d, R}}{|\log \varepsilon|}((v,y)+\varepsilon)^{-\alpha/2}(\varepsilon-(v,y))^{-\alpha/2} \mathbf{1}_{(y\in \mathbb{S}_\varepsilon^{d-2}(0, R))}\ell_d(\dd y),
\]
for an appropriate choice of $k_{1, d,R}$. We may now continue the argument from \eqref{pickup} with the calculation 
\begin{equation}
|\log\varepsilon|U\rho_\varepsilon^{(1)} (x)= \frac{ k_{\alpha,d, R}  \pi^{(d-2)/2}}{\Gamma((d-2)/2)}  \int_{-\varepsilon}^{\varepsilon}  (u+\varepsilon)^{-1/2} (\varepsilon-u)^{-1/2}\dd u \int_0^R \Big(w+{|(v,x)-u|^2}\Big)^{\frac{1-d}{2}} w^\frac{d-3}{2}{\dd w}.
\label{breakinto2}
\end{equation}
Let us now define 
\[
I({R}, \varepsilon, x) = \int_0^R \Big(w+{|(v,x)-u|^2}\Big)^{\frac{\alpha-d}{2}} w^\frac{d-3}{2}{\dd w} 
\]
ensuring that $\varepsilon$ is small enough that  $\varepsilon\ll R$. 
\smallskip

Appealing to \eqref{poly}, 
\begin{align*}
I_1({R}, \varepsilon, x)&= \frac{ k_{\alpha,d, R}  \pi^{(d-2)/2}}{\Gamma((d-2)/2)}  \int_{-\varepsilon}^{\varepsilon}  (u+\varepsilon)^{-1/2} (\varepsilon-u)^{-1/2}\dd u \int_0^{R} \Big(w+{|(v,x)-u|^2}\Big)^{\frac{1-d}{2}} w^\frac{d-3}{2}{\dd w}\\
&=\frac{ k_{\alpha,d, R}  \pi^{(d-2)/2}}{\Gamma((d-2)/2)}  \int_{-\varepsilon}^{\varepsilon}  (\varepsilon^2 - u^2)^{-1/2} |(v,x)-u|^{1-d}\dd u \int_0^{R} \Big(\frac{w}{|(v,x)-u|^2}+{1}\Big)^{\frac{1-d}{2}} w^\frac{d-3}{2}{\dd w}\\
&=\frac{ k_{\alpha,d, R}  \pi^{(d-2)/2}}{\Gamma((d-2)/2)}  \int_{-\varepsilon}^{\varepsilon}  (\varepsilon^2 - u^2)^{-1/2} |(v,x)-u|^{1-d}\notag\\
&\hspace{4cm}
\frac{{R}^{(d-1)/2}}{(d-1)/2}
 {_2}F_1\left(\frac{d-1}{2}, \frac{d-1}{2}; \frac{d+1}{2}; -\frac{{R}}{|(v,x)-u|^2}\right)
\dd u, \\
\end{align*}
where we have used the identity in \eqref{third}. One of the many identities for hypergeometric functions, see \cite{webpage2}, offers us the growth condition, for $c-a\in\mathbb{N}$, as $|z|\to\infty$,
\begin{align}
{_2}F_1(a, a; c; z) \sim
\frac{\Gamma(c)(\log(-z)-\psi(c-a)-\psi(a)-2\gamma)(-z)^{-a}}{\Gamma(a)(c-a-1)!} + \frac{\Gamma(c)2(-z)^{-c}}{\Gamma(a)^2((c-a)!)^2}
\end{align}
where $\gamma$ is an unimportant constant and $\psi(z) = \Gamma'(z)/\Gamma(z)$ is the di-gamma function. In the spirit of previous calculations, we can thus find to leading order, uniformly over $x\in \mathsf{D}_\varepsilon$, 
\begin{equation}
U\rho_\varepsilon^{(1)} (x)\sim 2\frac{ \pi^{d/2}k_{\alpha,d, R}  }{\Gamma((d-2)/2)},   
\label{Tsogiigotfactor2different}
\end{equation}
which remarkably does not depend on $R$. This means we should choose the constant 
\[
k_{\alpha,d, R}  = \frac{\Gamma((d-2)/2)}{2\pi^{d/2}} 
\]
 for this asymptotic to serve our purpose. 
\smallskip

At this point in the proof, recalling the fundamental decomposition \eqref{differencee3}, it is worth bringing in the term  $U\mu^{(2)}_\varepsilon$ and noting that one can compute with relatively coarse estimates that 
\[
\sup_{x\in \mathsf{D}_\varepsilon}\left|U\rho^{(2)}_\varepsilon(x)\right|\leq \frac{C}{ |\log\varepsilon|}
\]
for some unimportant constant $C>0$. Together with \eqref{Tsogiigotfactor2different}, in a calculation similar to \eqref{somethingsimilarbelow} we can put the pieces together to get the  asymptotic, for $x\not\in\mathsf{D}$ and $\varepsilon$ sufficiently small, 
\begin{align}
\lim_{\varepsilon\to0}\ |\log\varepsilon|\ \mathbb{P}_x(\tau_{\mathsf{D}_\varepsilon}<\infty) 
&= \lim_{\varepsilon\to0} \ |\log\varepsilon|\ U\rho_\varepsilon^\mathsf{D}(x) \notag\\
&= \lim_{\varepsilon\to0} \frac{\Gamma((d-2)/2)}{2\pi^{d/2}} \int_{\mathsf{D}_\varepsilon}|x-y|^{1 -d}(\varepsilon^2-(v,y)^2)^{-1/2} \ell_d(\dd y)\notag\\
&= \lim_{\varepsilon\to0} \frac{\Gamma((d-2)/2)}{\pi^{d/2}} \int_{-\varepsilon}^\varepsilon(\varepsilon^2-u^2)^{-1/2} \dd u \int_{\mathsf{D}(u)}|x-\hat{y}|^{1-d}\ell_{d-1}(\dd \hat y)\notag\\
&=\frac{\Gamma((d-2)/2)}{\pi^{(d-2)/2}}M_\mathsf{D}(x),
\end{align}
where    $\mathsf{D} (u):= \{y \in \mathbb{H}^{d-1} (u)\colon \hat{y}\in \mathsf{D} \}$. The proof is complete.
\qed


\appendix
\renewcommand{\theequation}{A.\arabic{equation}}
\setcounter{equation}{0}
\section*{Appendix: Hypergeometric identities}

We work with the standard definition for the hypergeometric function,
\[
\,{_2}F_1(a,b,c; z) = \sum_{n = 0}^\infty\frac{(a)_n(b)_n}{(c)_n} \frac{z^n}{n!}, \qquad |z|<1.
\]
Of the many identities for hypergeometric functions, we need the following:
\begin{align}
    {_2}F_1 (a,b,c;z) &= \frac{\Gamma(c)\Gamma(a+b-c)}{\Gamma(a)\Gamma(b)} (1-z)^{c-a-b} 
      {_2}F_1 (c-a,c-b,1+c-a-b; 1-z)
    \notag \\
    &\hspace{2cm}+ \frac{\Gamma(c)\Gamma(c-a-b)}{\Gamma(c-a)\Gamma(c-b)}
      {_2}F_1 (a,b,a+b-c+1; 1-z)
  \label{bigF}
\end{align}

for $c-a-b \notin \mathbb{Z},$ and hence thanks to continuity, 
\begin{align}
&\lim_{\varepsilon\to0} \sup_{r\in[1-\varepsilon,1]} \Bigg|  {_2}F_1 \Big(\frac{d-\alpha}{2}, 1-\frac{\alpha}{2}; \frac{d}{2}; r^2\Big)\notag\\
&\hspace{3cm} - 
    \frac{\Gamma({d}/{2})\Gamma(1-\alpha)}{\Gamma((d-\alpha)/2) \Gamma((2-{\alpha})/{2})} (1-r^2)^{\alpha-1}  -\frac{\Gamma({d}/{2})\Gamma(\alpha-1)}{\Gamma({\alpha}/{2})\Gamma((d+\alpha-2)/2)} \Bigg|=0  \label{f1}.
\end{align}

We will need to apply a similar identity to \eqref{bigF} but for the setting that $c-a-b = 0$, which violates the assumption behind \eqref{bigF}. In that case, we need to appeal to the formula 
\begin{eqnarray}
    {_2}F_1 (a,b,a+b,z) &=& \frac{\Gamma(a+b)}{\Gamma(a)\Gamma(b)} \Big( \sum_{k=0}^\infty \frac{(a)_k (b)_k}{(k!)^2} (2\psi(k+1)-\psi(a+k)-\psi(b+k)) (1-z)^k \notag \\
    &-& \log(1-z)\, {_2}F_1 (a,b,1,1-z)\Big).
\end{eqnarray}
for $|1-z|<1$ where the di-gamma function $\psi(z) = \Gamma'(z)/\Gamma(z)$ is defined for all $z \ne -n, n \in \mathbb{N}$. 

\smallskip
Again, thanks to continuity, we can write
\begin{align}
&\lim_{\varepsilon\to0} \sup_{r\in[1-\varepsilon,1]} \Bigg|  {_2}F_1 \Big(\frac{d-1}{2}, \frac{1}{2}; \frac{d}{2}; r^2\Big)
    +\frac{\Gamma({d}/{2})}{\Gamma((d-1)/2) \Gamma(1/{2})} \log(1-r^2)\notag\\
&\hspace{6cm}    -\frac{2\Gamma({d}/{2}) (\psi(1) - \psi((d-1)/2) -\psi(1/{2}))}{\Gamma((d-1)/2) \Gamma(1/{2})}  \Bigg|=0,  \label{f12}
\end{align}

\medskip

A second identity that is needed is the following combination formula, which states that for any $|z|<1$, we have
\begin{align}
    {_2}F_1(a,b;c;z) &= \frac{\Gamma(b-a)\Gamma(c)}{\Gamma(c-a)\Gamma(b)} (-z)^{-a} {_2}F_1\left(a,a-c+1;a-b+1;\frac{1}{z}\right) \notag \\
    &\hspace{2cm}+ \frac{\Gamma(a-b)\Gamma(c)}{\Gamma(c-b)\Gamma(a)} (-z)^{-b} {_2}F_1\left(b-c+1,b;-a+b+1;\frac{1}{z}\right) 
    \label{linearcombo}
\end{align}
which can be found, for example at \cite{webpage}. In the main body of the text, we use this identity for the setting that $a = \alpha/2$, $b = \alpha$ and $c = 1+\alpha/2$. This gives us the identity 
\begin{align}
    {_2}F_1\left(\frac{\alpha}{2},\alpha;1+\frac{\alpha}{2};z\right) &= \frac{\Gamma(\alpha/2)\Gamma((2+\alpha)/2)}{ \Gamma(\alpha)} (-z)^{-\alpha/2} {_2}F_1\left(\alpha/2,0;1-\alpha/2;\frac{1}{z}\right) \notag \\
    &\hspace{2cm}+ \frac{\Gamma(-\alpha/2)\Gamma((2+\alpha)/2)}{\Gamma((2-\alpha)/2)\Gamma(\alpha/2)} (-z)^{-\alpha} {_2}F_1\left(\alpha/2,\alpha;1+\alpha/2;\frac{1}{z}\right) \notag\\
    &= \frac{\Gamma(\alpha/2)\Gamma((2+\alpha)/2)}{ \Gamma(\alpha)} (-z)^{-\alpha/2} \notag \\
    &\hspace{2cm}- (-z)^{-\alpha} {_2}F_1\left(\alpha/2,\alpha;1+\alpha/2;\frac{1}{z}\right) ,
        \notag
\end{align}
where we have used the recursion formula for gamma functions twice in the final equality. This allows us to come to rest at the following useful identity 
\begin{align}
  (-{z})^{-\alpha/2} {_2}F_1\left(\alpha/2,\alpha;1+\alpha/2;\frac{1}{z}\right) + (-z)^{\alpha/2}  {_2}F_1\left(\frac{\alpha}{2},\alpha;1+\frac{\alpha}{2};z\right) 
    &= \frac{\Gamma(\alpha/2)\Gamma((2+\alpha)/2)}{ \Gamma(\alpha)}.
            \label{linearcombo2}
        \end{align}

We are also interested in integral formulae, for which the hypergeometric function is used to evaluate an integral. The first  is aversion of formula 3.665(2) in \cite{table} which states that,  for any $0<|a|<r$ and $\nu>0$,  as
\begin{equation}
     \int_0^{\pi} \frac{\sin^{d-2} \phi }{(a^2+2a r \cos \phi + r^2)^\nu} \dd\phi=  \frac{1}{r^{2\nu}} B\Big(\frac{d-1}{2}, \frac{1}{2}\Big) \,{_2}F_1 \Big(\nu, \nu-\frac{d}{2}+1; \frac{d}{2}; \frac{a^2}{r^2}\Big) ,
     \label{3.665}
\end{equation}
where $B(a,b) = \Gamma(a)\Gamma(b)/\Gamma(a+b)$ is the Beta function.
The second is  formula 3.197.8 in \cite{table}, which states that, for $Re(\mu)>0, Re(\nu)>0$ and $|\arg({u}/{\beta})|<\pi$,  
\begin{equation}
    \int_0^u x^{\nu-1}(u-x)^{\mu-1}  (x+\beta)^\lambda \dd x = \beta^\lambda u^{\mu+\nu-1} B(\mu,\nu) {_2}F_1 \left(-\lambda,\nu;\mu+\nu;-\frac{u}{\beta}\right).
    \label{poly}
\end{equation}
    The third is 3.194.1 of \cite{table} and states that, for $|\arg(1+\beta u)|>\pi$ and $Re(\mu)>0$, $Re(\nu)>0$,
    \begin{equation}
    \int_0^u x^{\mu -1} (1+\beta x)^{-\nu}\dd x = \frac{u^\mu}{\mu}{_2}F_1(\nu, \nu-\mu; 1+\mu; -\beta u),
    \label{third}
    \end{equation}
where ${_2}F_1$ in the above identity is understood as its analytic extension in the event that $|\beta u|>1$.

\section*{Acknowledgements}
TS acknowledges support from a Schlumberger Faculty of the Future award. SP acknowledges support from the Royal Society as a Newton International Fellow Alumnus (AL201023) and UNAM-DGAPA-PAPIIT grant no. IA103220.

\end{document}